\documentclass[a4, 12pt]{amsart}
\usepackage[mathscr]{eucal}
\usepackage{amssymb}
\usepackage{latexsym}
\usepackage{amsthm}
\usepackage{color}
\theoremstyle{plain}
\newtheorem{theorem}{Theorem}[section]

\newtheorem{remark}{Remark}[section]
\newtheorem{lemma}{Lemma}[section]

\makeatletter
\@addtoreset{equation}{section}

\title[Complete selfshrinkers ]
{Complete self-shrinkers with constant norm of the second fundamental form}
\author [Q. -M. Cheng, Z. Li and G. Wei]{Qing-Ming Cheng, Zhi Li  and Guoxin Wei}
\address{Qing-Ming Cheng \\  \newline \indent Department of Applied Mathematics, Faculty of Sciences,
\newline \indent Fukuoka University, Fukuoka  814-0180, Japan.  \newline \indent cheng@fukuoka-u.ac.jp}
\address{Zhi Li \\  School of Mathematical Sciences, South China Normal University,
\newline \indent 510631, Guangzhou,  China.  \newline \indent lizhihnsd@126.com}
\address{Guoxin Wei \\  School of Mathematical Sciences, South China Normal University,
\newline \indent 510631, Guangzhou,  China. \newline \indent  weiguoxin@tsinghua.org.cn}

\begin{document}
\maketitle

\begin{abstract}
In this paper, we classify $3$-dimensional complete self-shrinkers in Euclidean space $\mathbb R^{4}$  with constant squared norm of the
second fundamental form $S$ and constant $f_{4}$.

\end{abstract}


\footnotetext{The first author was partially  supported by JSPS Grant-in-Aid for Scientific Research (B):  No.16H03937.
The third  author was partly supported by grant No. 11771154 of NSFC and by GDUPS (2018).}

\section{introduction}
\vskip2mm
\noindent
\noindent
One of the most important problems in   mean curvature flow is to understand
the possible singularities that the flow goes through.  A key starting point
for singularity analysis is Huisken's monotonicity formula. The monotonicity
implies that the solution to the flow is asymptotically self-similar near a given type I singularity. Thus, it
 is modeled  by  self-shrinking solutions of the flow.
An  $n$-dimensional  submanifold  $X: M\rightarrow \mathbb{R}^{n+p}$  in the $(n+p)$-dimensional
Euclidean space $\mathbb{R}^{n+p}$  is called a self-shrinker if it satisfies
\begin{equation*}
\vec H+ X^{\perp}=0,
\end{equation*}
where  $X^{\perp}$  and $\vec H$ denote the normal part of the position vector $X$ and mean curvature vector of  this
submanifold.
It is known that self-shrinkers play an important role in the study on singularities  of the mean curvature flow because
they describe all possible  blow-ups at a given singularity.   \newline
For the classification of complete self-shrinkers with co-dimension $1$,  many nice works were done. Abresch and Langer \cite{AL}  classified
closed self-shrinkering curves in $\mathbb{R}^2$ and showed that the  round circle is the only embedded self-shrinker.
Huisken \cite{H2, H3},  Colding and Minicozzi \cite{CM}  classified $n$-dimensional  complete  embedded self-shrinkers
in $\mathbb{R}^{n+1}$ with  mean curvature $H\geq 0$ and  with polynomial volume growth.
According to the results  of Halldorsson \cite{H},
Ding and Xin \cite{DX1}, Cheng and Zhou \cite{CZ}, one knows
that for any positive integer $n$,
$\Gamma \times \mathbb R^{n-1}$ is a complete self-shrinker without polynomial volume growth in $\mathbb R^{n+1}$, where
$\Gamma$ is a complete self-shrinking curve of Halldorsson  \cite{H}. Hence, the  condition of polynomial volume growth in \cite{H3} and \cite{CM}
is essential. Furthermore,  for the study on the rigidity of complete self-shrinkers, many important works have been done
(cf. \cite{CL},  \cite{CO},  \cite{CP}, \cite{CW}, \cite{DX1}, \cite{DX2}, \cite{LW1}, \cite{LW2}  and so on). In particular, by estimating the first eigenvalue of the Dirichlet  eigenvalue problem,
Ding and Xin \cite{DX2}  studied $2$-dimensional complete self-shrinkers with polynomial volume growth. They  proved
that a $2$-dimensional complete self-shrinker  $X: M\rightarrow \mathbb{R}^{3}$  with polynomial volume growth
and with constant  squared norm $S$ of the second fundamental form
is isometric to one of
 $\mathbb{R}^{2}$,
 $S^{1} (1)\times \mathbb{R}$ and $S^{2}(\sqrt{2})$.

\noindent Recently,  Cheng and Ogata \cite{CO} have given a complete classification for 2-dimensional complete self-shrinkers with
 constant  squared norm $S$ of the second fundamental form, that is, they have proved the following:

\vskip3mm
\noindent
{\bf Theorem CO.}
{\it A $2$-dimensional complete self-shrinker  $X: M\rightarrow \mathbb{R}^{3}$  in $\mathbb{R}^{3}$ with constant  squared norm of the second fundamental form
is isometric to one of the following:
\begin{enumerate}
\item $\mathbb{R}^{2}$,
\item
 $S^1 (1)\times \mathbb{R}$,
\item  $S^{2}(\sqrt{2})$.
\end{enumerate}
}

\vskip2mm
\noindent
For the higher dimension $n$, it is not easy to classify self-shrinkers in Euclidean space with constant squared norm $S$. In this paper, under the assumption that $f_4$ constant, we give a complete classification for  $3$-dimensional complete  self-shrinker in $\mathbb R^{4}$ with constant squared norm $S$. In fact, we prove the following result.
\begin{theorem}\label{theorem 1}
 Let $X: M^{3}\to \mathbb{R}^{4}$ be a
$3$-dimensional complete  self-shrinker in $\mathbb R^{4}$.
If the squared norm $S$ of the second fundamental form and $f_{4}$ are constant, then $X: M^{3}\to \mathbb{R}^{4}$ is
isometric to one of
\begin{enumerate}
\item $\mathbb {R}^{3}$,
\item $S^{1}(1)\times \mathbb{R}^{2}$,
\item $S^{2}(\sqrt{2})\times \mathbb{R}^{1}$,
\item $S^{3}(\sqrt{3})$.
\end{enumerate}
In particular, $S$ must be $0$ and $1$; $f_4$ must be $0$, $\frac{1}{3}$, $\frac{1}{2}$ and $1$, where $S=\sum\limits_{i,j}h_{ij}^2$ and  $f_{4}=\sum\limits_{i,j,k,l}h_{ij}h_{jk}h_{kl}h_{li}$.
\end{theorem}

\begin{remark} It is well-known that in the some senses,  the behavior of complete self-shrinkers $X: M^{n}\to \mathbb{R}^{n+1}$  is similar to one of compact
minimal hypersurfaces in the unit spheres. Since Chern conjecture on 3-dimensional compact minimal hypersurfaces in the unit sphere was solved affirmatively, 
one wants to  give a complete classification for  $3$-dimensional complete  self-shrinker in $\mathbb R^{4}$ with constant squared norm $S$. It is very difficult.
In fact, for 3-dimensional minimal hypersurfaces in the unit sphere, $f_4=\frac12 S^2$. Hence, in this case, if $S$ is constant, then $f_4$ is also constant.  But for 
self-shrinkers, we do not have this property.  Hence,  in the proof of our theorem,  it plays an important role  that $f_4$ is constant.
\end{remark}

\vskip5mm
\section {Preliminaries}
\vskip2mm

\noindent
Let $X: M\rightarrow\mathbb{R}^{n+1}$ be an
$n$-dimensional connected hypersurface of the $n+1$-dimensional Euclidean space
$\mathbb{R}^{n+1}$. We choose a local orthonormal frame field
$\{e_A\}_{A=1}^{n+1}$ in $\mathbb{R}^{n+1}$ with dual coframe field
$\{\omega_A\}_{A=1}^{n+1}$, such that, restricted to $M$,
$e_1,\cdots, e_n$ are tangent to $M^n$.

\noindent
From now on,  we use the following conventions on the ranges of indices:
$$
 1\leq i,j,k,l\leq n
$$
and $\sum_{i}$ means taking  summation from $1$ to $n$ for $i$.
Then we have
\begin{equation*}
dX=\sum_i\limits \omega_i e_i,
\end{equation*}
\begin{equation*}
de_i=\sum_j\limits \omega_{ij}e_j+\omega_{i n+1}e_{n+1},
\end{equation*}
\begin{equation*}
de_{n+1}=\omega_{n+1 i}e_i,
\end{equation*}
where $\omega_{ij}$ is the Levi-Civita connection of the hypersurface.

\noindent By  restricting  these forms to $M$,  we get
\begin{equation}\label{2.1-1}
\omega_{n+1}=0
\end{equation}
and the induced Riemannian metric of the hypersurface  is written as
$ds^2_M=\sum_i\limits\omega^2_i$.
Taking exterior derivatives of \eqref{2.1-1}, we obtain
\begin{equation*}
0=d\omega_{n+1}=\sum_i \omega_{n+1 i}\wedge\omega_i.
\end{equation*}
By Cartan's lemma, we know
\begin{equation}\label{2.1-2}
\omega_{in+1}=\sum_j h_{ij}\omega_j,\quad
h_{ij}=h_{ji}.
\end{equation}

$$
h=\sum_{i,j}h_{ij}\omega_i\otimes\omega_j
$$
and
$$
H= \sum_i\limits h_{ii}
$$
are called  the second fundamental form and the mean curvature  of $X: M\rightarrow\mathbb{R}^{n+1}$, respectively.
Let $S=\sum_{i,j}\limits (h_{ij})^2$ be  the squared norm
of the second fundamental form  of $X: M\rightarrow\mathbb{R}^{n+1}$.
The induced structure equations of $M$ are given by
\begin{equation*}
d\omega_{i}=\sum_j \omega_{ij}\wedge\omega_j, \quad  \omega_{ij}=-\omega_{ji},
\end{equation*}
\begin{equation*}
d\omega_{ij}=\sum_k \omega_{ik}\wedge\omega_{kj}-\frac12\sum_{k,l}
R_{ijkl} \omega_{k}\wedge\omega_{l},
\end{equation*}
where $R_{ijkl}$ denotes components of the curvature tensor of the hypersurface.
Hence,
the Gauss equations are given by
\begin{equation}\label{2.1-3}
R_{ijkl}=h_{ik}h_{jl}-h_{il}h_{jk}.
\end{equation}

\noindent
Defining the
covariant derivative of $h_{ij}$ by
\begin{equation}\label{2.1-4}
\sum_{k}h_{ijk}\omega_k=dh_{ij}+\sum_kh_{ik}\omega_{kj}
+\sum_k h_{kj}\omega_{ki},
\end{equation}
we obtain the Codazzi equations
\begin{equation}\label{2.1-5}
h_{ijk}=h_{ikj}.
\end{equation}
By taking exterior differentiation of \eqref{2.1-4}, and
defining
\begin{equation}\label{2.1-6}
\sum_lh_{ijkl}\omega_l=dh_{ijk}+\sum_lh_{ljk}\omega_{li}
+\sum_lh_{ilk}\omega_{lj}+\sum_l h_{ijl}\omega_{lk},
\end{equation}
we have the following Ricci identities:
\begin{equation}\label{2.1-7}
h_{ijkl}-h_{ijlk}=\sum_m
h_{mj}R_{mikl}+\sum_m h_{im}R_{mjkl}.
\end{equation}
Defining
\begin{equation}\label{2.1-8}
\begin{aligned}
\sum_mh_{ijklm}\omega_m&=dh_{ijkl}+\sum_mh_{mjkl}\omega_{mi}
+\sum_mh_{imkl}\omega_{mj}+\sum_mh_{ijml}\omega_{mk}\\
&\ \ +\sum_mh_{ijkm}\omega_{ml}
\end{aligned}
\end{equation}
and taking exterior differentiation of  \eqref{2.1-6}, we get
\begin{equation}\label{2.1-9}
\begin{aligned}
h_{ijkln}-h_{ijknl}&=\sum_{m} h_{mjk}R_{miln}
+ \sum_{m}h_{imk}R_{mjln}+ \sum_{m}h_{ijm}R_{mkln}.
\end{aligned}
\end{equation}
For a smooth function $f$, we define
\begin{equation}\label{2.1-10}
\sum_i f_{,i}\omega_i=df,
\end{equation}
\begin{equation}\label{2.1-11}
\sum_j f_{,ij}\omega_j=df_{,i}+\sum_j
f_{,j}\omega_{ji},
\end{equation}
\begin{equation}\label{2.1-12}
|\nabla f|^2=\sum_{i }(f_{,i})^2,\ \ \ \  \Delta f =\sum_i f_{,ii}.
\end{equation}
The $\mathcal{L}$-operator is defined by
\begin{equation*}
\mathcal{L}f=\Delta f-\langle X,\nabla f\rangle,
\end{equation*}
where $\Delta$ and $\nabla$ denote the Laplacian and the gradient
operator, respectively.
\vskip2mm
\noindent
Formulas in the following Lemma 2.1 can be found in  \cite{CW}.
\begin{lemma}
Let $X:M^n\rightarrow \mathbb{R}^{n+1}$ be an $n$-dimensional complete self-shrinker in $\mathbb R^{n+1}$. We have
\begin{equation}\label{2.1-13}
\mathcal{L}H=H(1-S).
\end{equation}
\begin{equation}\label{2.1-14}
\aligned
\frac{1}{2}\mathcal{L}
|X|^{2}=n-|X|^{2}.
\endaligned
\end{equation}
\begin{equation}\label{2.1-15}
\frac{1}{2}\mathcal{L}S
=\sum_{i,j,k}h_{ijk}^{2}+(1-S)S,
\end{equation}
where $S=\sum\limits_{i,j}h_{ij}^2$.
\end{lemma}

\noindent
\begin{lemma}
Let $X:M^{n}\rightarrow \mathbb{R}^{n+1}$ be  a $n$-dimensional complete self-shrinker in $\mathbb R^{n+1}$. If $S$ is constant,  we have
\begin{equation}\label{2.1-16}
\aligned
\frac{1}{2}\mathcal{L}\sum_{i, j,k}(h_{ijk})^{2}
=&\sum_{i,j,k,l}(h_{ijkl})^{2}+(2-S)\sum_{i,j,k}(h_{ijk})^{2}+6\sum_{i,j,k,l,p}h_{ijk}h_{il}h_{jp}h_{klp}\\
&-3\sum_{i,j,k,l,p}h_{ijk}h_{ijl}h_{kp}h_{lp},
\endaligned
\end{equation}
\end{lemma}
\begin{proof} By making use of the Ricci identities  \eqref{2.1-7},  \eqref{2.1-9} and a direct calculation, we
can  obtain \eqref{2.1-16}.
\end{proof}

\noindent
We define two functions $f_3$  and $f_4$ as follows:
$$
f_3=\sum_{i,j,k}h_{ij}h_{jk}h_{ki},
$$
$$
 \ f_4=\sum_{i,j,k,l}h_{ij}h_{jk}h_{kl}h_{li},
$$
then we get the following result.
\begin{lemma}\label{lemma 2.3}
Let $X:M^{3}\rightarrow \mathbb{R}^{4}$ be  a $n$-dimensional complete hypersurface in $\mathbb R^{4}$. Then  we can choose a local field of orthonormal frames on $M^3$ such that, at the point,
$h_{ij}=\lambda_i\delta_{ij}$,

$$f_{3}=\frac{H}{2}(3S-H^{2})+3\lambda_{1}\lambda_{2}\lambda_{3},$$
$$f_{4}=\frac{4}{3}Hf_{3}-H^{2}S+\frac{1}{6}H^{4}+\frac{1}{2}S^{2},$$

$$
\nabla_{l}f_{3}=3\sum_{i,j,k}h_{ijl}h_{jk}h_{ki}, \ \ \text{for } \ l=1, 2, 3,
$$
$$
\nabla_{p}\nabla_{l}f_{3}=3\sum_{i,j,k}h_{ijlp}h_{jk}h_{ki}+6\sum_{i,j,k}h_{ijl}h_{jkp}h_{ki}, \ \ \text{for } \ l,p=1, 2, 3.
$$
and
$$
\nabla_{m}f_{4}=4\sum_{i,j,k,l}h_{ijm}h_{jk}h_{kl}h_{li}, \ \ \text{for } \ m=1, 2, 3,
$$

\begin{equation*}
\begin{aligned}
\nabla_{p}\nabla_{m}f_{4}=&4\sum_{i,j,k,l}h_{ijmp}h_{jk}h_{kl}h_{li}\\
                          &+4\sum_{i,j,k,l}h_{ijm}(2h_{jkp}h_{kl}h_{li}+h_{jk}h_{klp}h_{li}), \  \text{for } \ m,p=1, 2, 3.
\end{aligned}
\end{equation*}

\begin{equation}\label{2.1-17}
\nabla_{k}f_{4}=\frac{4}{3}f_{3} H_{,k}+\frac{4}{3}H\nabla_{k}f_{3}-2SHH_{,k}+\frac{2}{3}H^{3}H_{,k},
\end{equation}

\begin{equation}\label{2.1-18}
\begin{aligned}
\nabla_{l}\nabla_{k}f_{4} =&\frac{4}{3}f_{3}H_{,kl}-2SHH_{,kl}
   +\frac{2}{3}H^{3}H_{,kl}+\frac{4}{3}H\nabla_{l}\nabla_{k}f_{3}+\frac{4}{3}\nabla_{l}f_{3} H_{,k} \\
  &+\frac{4}{3}H_{,l}\nabla_{k}f_{3}-2SH_{,k}H_{,l}+2H^{2}H_{,k}H_{,l},
\end{aligned}
\end{equation}

\begin{equation}\label{2.1-19}
\begin{aligned}
&\nabla_{m}\nabla_{l}\nabla_{k}f_{4}\\
=&\frac{4}{3}f_{3}H_{,klm}-2SHH_{,klm}
  +\frac{2}{3}H^{3}H_{,klm}+\frac{4}{3}H\nabla_{m}\nabla_{l}\nabla_{k}f_{3}+\frac{4}{3}\nabla_{m}f_{3}H_{,kl}\\
 &-2SH_{,m}H_{,kl}+2H^{2}H_{,m}H_{,kl}+\frac{4}{3}H_{,m}\nabla_{l}\nabla_{k}f_{3}+\frac{4}{3}\nabla_{l}f_{3} H_{,km} \\
 &+\frac{4}{3}\nabla_{m}\nabla_{l}f_{3} H_{,k}+\frac{4}{3}H_{,l}\nabla_{m}\nabla_{k}f_{3}+\frac{4}{3}H_{,lm}\nabla_{k}f_{3}-2SH_{,km}H_{,l} \\
 &-2SH_{,k}H_{,lm}+2H^{2}H_{,km}H_{,l}+2H^{2}H_{,k}H_{,lm}+4HH_{,k}H_{,l}H_{,m},
\end{aligned}
\end{equation}

\begin{equation}\label{2.1-20}
\begin{aligned}
&\nabla_{n}\nabla_{m}\nabla_{l}\nabla_{k}f_{4}\\
 =&\frac{4}{3}f_{3}H_{,klmn}+\frac{4}{3}\nabla_{n}f_{3}H_{,klm}-2SHH_{,klmn}-2SH_{,n}H_{,klm}+\frac{2}{3}H^{3}H_{,klmn} \\
  &+2H^{2}H_{,n}H_{,klm}+\frac{4}{3}H\nabla_{n}\nabla_{m}\nabla_{l}\nabla_{k}f_{3}+\frac{4}{3}H_{,n}\nabla_{m}\nabla_{l}\nabla_{k}f_{3}
   +\frac{4}{3}\nabla_{m}f_{3}H_{,kln}\\
  &+\frac{4}{3}\nabla_{n}\nabla_{m}f_{3}H_{,kl}-2SH_{,m}H_{,kln}-2SH_{,mn}H_{,kl}+2H^{2}H_{,m}H_{,kln} \\
  &+2H^{2}H_{,mn}H_{,kl}+4HH_{,n}H_{,m}H_{,kl}+\frac{4}{3}H_{,m}\nabla_{n}\nabla_{l}\nabla_{k}f_{3}+\frac{4}{3}H_{,mn}\nabla_{l}\nabla_{k}f_{3}\\
  &+\frac{4}{3}\nabla_{l}f_{3} H_{,kmn}+\frac{4}{3}\nabla_{n}\nabla_{l}f_{3} H_{,km}+\frac{4}{3}\nabla_{n}\nabla_{m}\nabla_{l}f_{3} H_{,k}+\frac{4}{3}\nabla_{m}\nabla_{l}f_{3} H_{,kn}\\
  &+\frac{4}{3}H_{,l}\nabla_{n}\nabla_{m}\nabla_{k}f_{3}+\frac{4}{3}H_{,ln}\nabla_{m}\nabla_{k}f_{3}+\frac{4}{3}H_{,lmn}\nabla_{k}f_{3}
  +\frac{4}{3}H_{,lm}\nabla_{n}\nabla_{k}f_{3} \\
  &-2SH_{,kmn}H_{,l}-2SH_{,km}H_{,ln}-2SH_{,k}H_{,lmn}-2SH_{,kn}H_{,lm}+2H^{2}H_{,km}H_{,ln} \\
  &+2H^{2}H_{,kmn}H_{,l}+4HH_{,km}H_{,l}H_{,n}+2H^{2}H_{,k}H_{,lmn}+2H^{2}H_{,kn}H_{,lm} \\
  &+4HH_{,n}H_{,k}H_{,lm}+4HH_{,k}H_{,l}H_{,mn}+4HH_{,k}H_{,ln}H_{,m}+4HH_{,kn}H_{,l}H_{,m} \\
  &+4H_{n}H_{,k}H_{,l}H_{,m},
\end{aligned}
\end{equation}
for k, l, m, n=1, 2, 3.
\end{lemma}

\noindent
We need the following lemma due to Cheng and Peng \cite{CP} to prove our result.

\noindent
\begin{lemma}
For an $n$-dimensional complete self-shrinker
 $X:M^{n}\rightarrow \mathbb{R}^{n+1} $  with   $\inf H^{2}>0$,
if  the squared norm $S$ of the second fundamental form is constant, then $M^{n}$ is  isometric to either
$S^{n}(\sqrt{n})$ or $S^{m}(\sqrt m)\times\mathbb{R}^{n-m}$ in $\mathbb{R}^{n+1}$, $1\leq m\leq n-1$.
\end{lemma}

 \vskip10mm
\section{Proof of the main result}

\vskip2mm
\noindent
From \eqref{2.1-15} of the Lemma 2.1, one has either $S=0$, or $S\geq1$. If $S=0$, then we know that $X: M^{3}\to \mathbb{R}^{4}$ is $\mathbb{R}^{3}$. Next, we assume that $S\geq 1$. From the Lemma 2.3, it is sufficient to prove that $\inf H^{2}>0$. We now prove the following theorem.

\begin{theorem}\label{theorem 2}
For a $3$-dimensional complete self-shrinker $X:M^{3}\rightarrow \mathbb{R}^{4}$ with non-zero constant squared norm $S$ of the second fundamental form and constant $f_{4}$, then  $\inf H^{2}>0$,
where $S=\sum_{i,j}h_{ij}^2$ and $f_{4}=\sum_{i,j,k,l}h_{ij}h_{jk}h_{kl}h_{li}$.
\end{theorem}

\begin{proof}
If $\inf H^{2}=0$, there exists a sequence $\{p_{t}\}$ in $M^{3}$ such that
\begin{equation*}
\lim_{t\rightarrow\infty} H^{2}(p_{t})=\inf H^{2}=\bar H^2=0.
\end{equation*}

\noindent From \eqref{2.1-15},
\eqref{2.1-16} and $S$ being constant, we know that
$\{h_{ij}(p_{t})\}$,  $\{h_{ijk}(p_{t})\}$ and $\{h_{ijkl}(p_{t})\}$ are bounded sequences, one can assume
$$\lim_{t\rightarrow\infty}h_{ij}(p_{t})=\bar h_{ij}=\bar \lambda_i\delta_{ij}, \quad  \lim_{t\rightarrow\infty}h_{ijk}(p_{t})=\bar h_{ijk}, \quad \lim_{t\rightarrow\infty}h_{ijkl}(p_{t})=\bar h_{ijkl}$$
for $i, j, k, l=1, 2, 3$.
Then,
$$
\bar H=\sum_i \bar h_{ii}=\bar \lambda_{1}+\bar \lambda_{2}+\bar \lambda_{3}=0, \ S=\sum_{i,j}\bar h_{ij}^2=\bar \lambda^2_{1}+\bar \lambda^2_{2}+\bar \lambda^2_{3}=2(\bar \lambda^2_{1}+\bar \lambda^2_{2}+\bar \lambda_{1}\bar \lambda_{2}).
$$

\noindent From
$$H_{,i}=\sum_{k}h_{ik}\langle X, e_k\rangle, \ \ \text{\rm for} \ \ i=1, 2, 3,$$
we have
\begin{equation}\label{3.1-1}
\bar h_{11k}+\bar h_{22k}+\bar h_{33k}=\bar \lambda_{k}\lim_{t\rightarrow\infty} \langle X, e_{k} \rangle(p_{t}),\ \ \text{\rm for} \ \ k=1, 2, 3.
\end{equation}
Since
\begin{equation}\label{3.1-2}
\aligned
\nabla_{j}\nabla_{i}H
 =&\sum_{k}h_{ijk}\langle X,e_{k}\rangle+h_{ij}-H\sum_{k}h_{ik}h_{kj},
\endaligned
\end{equation}
we conclude
\begin{equation*}
\bar H_{,ij}=\sum_{k}\bar h_{ijk}\lim_{t\rightarrow\infty} \langle X,e_{k} \rangle(p_{t}) + \bar \lambda_{i} \delta_{ij} - \bar H\bar \lambda_{i}\bar \lambda_{j}\delta_{ij},
\end{equation*}
that is,
\begin{equation}\label{3.1-3}
\begin{cases}
\begin{aligned}
&\bar h_{1111}+\bar h_{2211}+\bar h_{3311}=\sum_{k}\bar h_{11k}\lim_{t\rightarrow\infty} \langle X,e_{k} \rangle(p_{t})+\bar \lambda_{1},\\
&\bar h_{1122}+\bar h_{2222}+\bar h_{3322}=\sum_{k}\bar h_{22k}\lim_{t\rightarrow\infty} \langle X,e_{k} \rangle(p_{t})+\bar \lambda_{2},\\
&\bar h_{1133}+\bar h_{2233}+\bar h_{3333}=\sum_{k}\bar h_{33k}\lim_{t\rightarrow\infty} \langle X,e_{k} \rangle(p_{t})+\bar \lambda_{3},\\
&\bar h_{1112}+\bar h_{2212}+\bar h_{3312}=\sum_{k}\bar h_{12k}\lim_{t\rightarrow\infty} \langle X,e_{k} \rangle(p_{t}),\\
&\bar h_{1113}+\bar h_{2213}+\bar h_{3313}=\sum_{k}\bar h_{13k}\lim_{t\rightarrow\infty} \langle X,e_{k} \rangle(p_{t}),\\
&\bar h_{1123}+\bar h_{2223}+\bar h_{3323}=\sum_{k}\bar h_{23k}\lim_{t\rightarrow\infty} \langle X,e_{k} \rangle(p_{t}).
\end{aligned}
\end{cases}
\end{equation}
Since $S$ is constant, we know
$$
\sum_{i,j}h_{ij}h_{ijk}=0, \ \ \text{for } \ k=1, 2, 3,
$$
$$
\sum_{i,j}h_{ij}h_{ijkl}+\sum_{i,j}h_{ijk}h_{ijl}=0, \ \ \text{for } \ k,l=1, 2, 3.
$$
Thus,
$$
\sum_{i,j}\bar h_{ij}\bar h_{ijk}=0, \ \ \text{for } \ k=1, 2, 3,
$$
$$
\sum_{i,j}\bar h_{ij}\bar h_{ijkl}+\sum_{i,j}\bar h_{ijk}\bar h_{ijl}=0, \ \ \text{for } \ k,l=1, 2, 3.
$$
Specifically,
\begin{equation}\label{3.1-4}
\bar\lambda_{1}\bar h_{11k}+\bar\lambda_{2}\bar h_{22k}+\bar\lambda_{3}\bar h_{33k}=0, \ \  \text{for } \ k=1, 2, 3,
\end{equation}
\begin{equation}\label{3.1-5}
\begin{cases}
\begin{aligned}
\bar \lambda_{1}\bar h_{1111}+\bar \lambda_{2}\bar h_{2211}+\bar \lambda_{3}\bar h_{3311}
=&-\bar h^{2}_{111}-\bar h^{2}_{221}-\bar h^{2}_{331}-2\bar h^{2}_{121}\\
   &-2\bar h^{2}_{131}-2\bar h^{2}_{231},\\
\bar \lambda_{1}\bar h_{1122}+\bar \lambda_{2}\bar h_{2222}+\bar \lambda_{3}\bar h_{3322}
=&-\bar h^{2}_{112}-\bar h^{2}_{222}-\bar h^{2}_{332}-2\bar h^{2}_{122}\\
   &-2\bar h^{2}_{132}-2\bar h^{2}_{232},\\
\bar \lambda_{1}\bar h_{1133}+\bar \lambda_{2}\bar h_{2233}+\bar \lambda_{3}\bar h_{3333}
=&-\bar h^{2}_{113}-\bar h^{2}_{223}-\bar h^{2}_{333}-2\bar h^{2}_{123}\\
  &-2\bar h^{2}_{133}-2\bar h^{2}_{233},\\
\bar \lambda_{1}\bar h_{1112}+\bar \lambda_{2}\bar h_{2212}+\bar \lambda_{3}\bar h_{3312}
=&-\bar h_{111}\bar h_{112}-\bar h_{221}\bar h_{222}-\bar h_{331}\bar h_{332}\\
    &-2\bar h_{121}\bar h_{122}-2\bar h_{131}\bar h_{132}-2\bar h_{231}\bar h_{232},\\
\bar \lambda_{1}\bar h_{1113}+\bar \lambda_{2}\bar h_{2213}+\bar \lambda_{3}\bar h_{3313}
=&-\bar h_{111}\bar h_{113}-\bar h_{221}\bar h_{223}-\bar h_{331}\bar h_{333}\\
   &-2\bar h_{121}\bar h_{123}-2\bar h_{131}\bar h_{133}-2\bar h_{231}\bar h_{233},\\
\bar \lambda_{1}\bar h_{1123}+\bar \lambda_{2}\bar h_{2223}+\bar \lambda_{3}\bar h_{3323}
=&-\bar h_{112}\bar h_{113}-\bar h_{222}\bar h_{223}-\bar h_{332}\bar h_{333}\\
 &-2\bar h_{122}\bar h_{123}-2\bar h_{132}\bar h_{133}-2\bar h_{232}\bar h_{233}.
\end{aligned}
\end{cases}
\end{equation}
From Ricci identities \eqref{2.1-7}, we obtain
\begin{equation*}
\bar h_{ijkl}-\bar h_{ijlk}=\bar\lambda_{i}\bar\lambda_{j}\bar\lambda_{k}\delta_{il}\delta_{jk}-\bar\lambda_{i}\bar\lambda_{j}\bar\lambda_{l}\delta_{ik}\delta_{jl}
+\bar\lambda_{i}\bar\lambda_{j}\bar\lambda_{k}\delta_{ik}\delta_{jl}-\bar\lambda_{i}\bar\lambda_{j}\bar\lambda_{l}\delta_{il}\delta_{jk},
\end{equation*}
that is,
\begin{equation}\label{3.1-6}
\begin{cases}
\begin{aligned}
&\bar h_{1212}-\bar h_{1221}=\bar \lambda_{1}\bar \lambda_{2}(\bar \lambda_{1}-\bar \lambda_{2}),\ \ \bar h_{1313}-\bar h_{1331}=\bar \lambda_{1}\bar \lambda_{3}(\bar \lambda_{1}-\bar \lambda_{3}),\\
&\bar h_{2323}-\bar h_{2332}=\bar \lambda_{2}\bar \lambda_{3}(\bar \lambda_{2}-\bar \lambda_{3}),\ \ \bar h_{iikl}-\bar h_{iilk}=0, \ \ \text{for} \ i,k,l=1, 2, 3.
\end{aligned}
\end{cases}
\end{equation}

\noindent From the lemma \ref{lemma 2.3}, we have

\begin{equation}\label{3.1-7}
\lim_{t\rightarrow\infty}\nabla_{k}f_{3}(p_{t})=3\bar \lambda^{2}_{1}\bar h_{11k}+3\bar \lambda^{2}_{2}\bar h_{22k}+3\bar \lambda^{2}_{3}\bar h_{33k}, \ \  \text{for } \ k=1, 2, 3,
\end{equation}

\begin{equation}\label{3.1-8}
\begin{cases}
\begin{aligned}
&\frac{1}{3}\lim_{t\rightarrow\infty}\nabla_{1}\nabla_{1}f_{3}(p_{t})\\
=&\bar \lambda^{2}_{1}\bar h_{1111}+\bar \lambda^{2}_{2}\bar h_{2211}+\bar \lambda^{2}_{3}\bar h_{3311}+2\bar \lambda_{1}(\bar h^{2}_{111}+\bar h^{2}_{121}+\bar h^{2}_{131})\\
&+2\bar \lambda_{2}(\bar h^{2}_{211}+\bar h^{2}_{221}+\bar h^{2}_{231})+2\bar \lambda_{3}(\bar h^{2}_{311}+\bar h^{2}_{321}+\bar h^{2}_{331}),\\

&\frac{1}{3}\lim_{t\rightarrow\infty}\nabla_{2}\nabla_{2}f_{3}(p_{t})\\
=&\bar \lambda^{2}_{1}\bar h_{1122}+\bar \lambda_{2}^{2}\bar h_{2222}+\bar \lambda^{2}_{3}\bar h_{3322}+2\bar \lambda_{1}(\bar h^{2}_{112}+\bar h^{2}_{122}+\bar h^{2}_{132})\\
&+2\bar \lambda_{2}(\bar h^{2}_{212}+\bar h^{2}_{222}+\bar h^{2}_{232})+2\bar \lambda_{3}(\bar h^{2}_{312}+\bar h^{2}_{322}+\bar h^{2}_{332}),\\

&\frac{1}{3}\lim_{t\rightarrow\infty}\nabla_{3}\nabla_{3}f_{3}(p_{t})\\
=&\bar \lambda^{2}_{1}\bar h_{1133}+\bar \lambda^{2}_{2}\bar h_{2233}+\bar \lambda^{2}_{3}\bar h_{3333}+2\bar \lambda_{1}(\bar h^{2}_{113}+\bar h^{2}_{123}+ \bar h^{2}_{133})\\
&+2\bar \lambda_{2}(\bar h^{2}_{213}+\bar h^{2}_{223}+\bar h^{2}_{233})+2\bar \lambda_{3}(\bar h^{2}_{313}+\bar h^{2}_{323}+\bar h^{2}_{333}),\\

&\frac{1}{3}\lim_{t\rightarrow\infty}\nabla_{2}\nabla_{1}f_{3}(p_{t})\\
=&\bar \lambda^{2}_{1}\bar h_{1112}+\bar \lambda^{2}_{2}\bar h_{2212}+\bar \lambda^{2}_{3}\bar h_{3312}+2\bar \lambda_{1}(\bar h_{111}\bar h_{112}+\bar h_{121}\bar h_{122}+\bar h_{131}\bar h_{132})\\
&+2\bar \lambda_{2}(\bar h_{211}\bar h_{212}+\bar h_{221}\bar h_{222}
+\bar h_{231}\bar h_{232})+2\bar \lambda_{3}(\bar h_{311}\bar h_{312}+\bar h_{321}\bar h_{322} \\
&+\bar h_{331}\bar h_{332}),\\

&\frac{1}{3}\lim_{t\rightarrow\infty}\nabla_{3}\nabla_{1}f_{3}(p_{t})\\
=&\bar \lambda^{2}_{1}\bar h_{1113}+\bar \lambda^{2}_{2}\bar h_{2213}+\bar \lambda^{2}_{3}\bar h_{3313}+2\bar \lambda_{1}(\bar h_{111}\bar h_{113}+\bar h_{121}\bar h_{123}+\bar h_{131}\bar h_{133})\\
&+2\bar \lambda_{2}(\bar h_{211}\bar h_{213}+\bar h_{221}\bar h_{223}
+\bar h_{231}\bar h_{233})+2\bar \lambda_{3}(\bar h_{311}\bar h_{313}+\bar h_{321}\bar h_{323} \\
&+\bar h_{331}\bar h_{333}),\\

&\frac{1}{3}\lim_{t\rightarrow\infty}\nabla_{3}\nabla_{2}f_{3}(p_{t})\\
=&\bar \lambda^{2}_{1}\bar h_{1123}+\bar \lambda^{2}_{2}\bar h_{2223}+\bar \lambda^{2}_{3}\bar h_{3323}+2\bar \lambda_{1}(\bar h_{112}\bar h_{113}+\bar h_{122}\bar h_{123}+\bar h_{132}\bar h_{133})\\
&+2\bar \lambda_{2}(\bar h_{212}\bar h_{213}+\bar h_{222}\bar h_{223}
+\bar h_{232}\bar h_{233})+2\bar \lambda_{3}(\bar h_{312}\bar h_{313}+\bar h_{322}\bar h_{323} \\
&+\bar h_{332}\bar h_{333}).
\end{aligned}
\end{cases}
\end{equation}
Since $f_{4}$ is constant, we know from the Lemma \ref{lemma 2.3},

\begin{equation}\label{3.1-9}
\bar\lambda^{3}_{1}\bar h_{11k}+\bar\lambda^{3}_{2}\bar h_{22k}+\bar\lambda^{3}_{3}\bar h_{33k}=0, \ \ \text{for } \ k=1, 2, 3,
\end{equation}

\begin{equation}\label{3.1-10}
\begin{cases}
\begin{aligned}
&\bar \lambda^{3}_{1}\bar h_{1111}+\bar \lambda^{3}_{2}\bar h_{2211}+\bar \lambda^{3}_{3}\bar h_{3311}\\
=&-3\bar \lambda^{2}_{1}\bar h^{2}_{111}-3\bar \lambda^{2}_{2}\bar h^{2}_{221}-3\bar \lambda^{2}_{3}
\bar h^{2}_{331}-2(\bar \lambda^{2}_{1}+\bar \lambda^{2}_{2}+\bar \lambda_{1}\bar \lambda_{2})\bar h^{2}_{121} \\
&-2(\bar\lambda^{2}_{1}+\bar\lambda^{2}_{3}+\bar\lambda_{1}\bar\lambda_{3})\bar h^{2}_{131}
-2(\bar\lambda^{2}_{2}+\bar\lambda^{2}_{3}+\bar\lambda_{2}\bar\lambda_{3})\bar h^{2}_{231},\\

&\bar \lambda^{3}_{1}\bar h_{1122}+\bar \lambda^{3}_{2}\bar h_{2222}+\bar \lambda^{3}_{3}\bar h_{3322}\\
=&-3\bar \lambda^{2}_{1}\bar h^{2}_{112}-3\bar \lambda^{2}_{2}\bar h^{2}_{222}-3\bar \lambda^{2}_{3}
\bar h^{2}_{332}-2(\bar \lambda^{2}_{1}+\bar \lambda^{2}_{2}+\bar \lambda_{1}\bar \lambda_{2})\bar h^{2}_{122} \\
&-2(\bar\lambda^{2}_{1}+\bar\lambda^{2}_{3}+\bar\lambda_{1}\bar\lambda_{3})\bar h^{2}_{132}
-2(\bar\lambda^{2}_{2}+\bar\lambda^{2}_{3}+\bar\lambda_{2}\bar\lambda_{3})\bar h^{2}_{232},\\

&\bar \lambda^{3}_{1}\bar h_{1133}+\bar \lambda^{3}_{2}\bar h_{2233}+\bar \lambda^{3}_{3}\bar h_{3333}\\
=&-3\bar \lambda^{2}_{1}\bar h^{2}_{113}-3\bar \lambda^{2}_{2}\bar h^{2}_{223}-3\bar \lambda^{2}_{3}
\bar h^{2}_{333}-2(\bar \lambda^{2}_{1}+\bar \lambda^{2}_{2}+\bar \lambda_{1}\bar \lambda_{2})\bar h^{2}_{123} \\
&-2(\bar\lambda^{2}_{1}+\bar\lambda^{2}_{3}+\bar\lambda_{1}\bar\lambda_{3})\bar h^{2}_{133}
-2(\bar\lambda^{2}_{2}+\bar\lambda^{2}_{3}+\bar\lambda_{2}\bar\lambda_{3})\bar h^{2}_{233},\\

&\bar \lambda^{3}_{1}\bar h_{1112}+\bar \lambda^{3}_{2}\bar h_{2212}+\bar \lambda^{3}_{3}\bar h_{3312}\\
=&-3\bar \lambda^{2}_{1}\bar h_{111}\bar h_{112}-3\bar \lambda^{2}_{2}\bar h_{221}\bar h_{222}-3\bar \lambda^{2}_{3}
\bar h_{331}\bar h_{332}-2(\bar \lambda^{2}_{1}+\bar \lambda^{2}_{2} \\
&+\bar \lambda_{1}\bar \lambda_{2})\bar h_{121}\bar h_{122}-2(\bar\lambda^{2}_{1}+\bar\lambda^{2}_{3}
+\bar\lambda_{1}\bar\lambda_{3})\bar h_{131}\bar h_{132}-2(\bar\lambda^{2}_{2}+\bar\lambda^{2}_{3} \\
&+\bar\lambda_{2}\bar\lambda_{3})\bar h_{231}\bar h_{232},\\

&\bar \lambda^{3}_{1}\bar h_{1113}+\bar \lambda^{3}_{2}\bar h_{2213}+\bar \lambda^{3}_{3}\bar h_{3313}\\
=&-3\bar \lambda^{2}_{1}\bar h_{111}\bar h_{113}-3\bar \lambda^{2}_{2}\bar h_{221}\bar h_{223}-3\bar \lambda^{2}_{3}
\bar h_{331}\bar h_{333}-2(\bar \lambda^{2}_{1}+\bar \lambda^{2}_{2} \\
&+\bar \lambda_{1}\bar \lambda_{2})\bar h_{121}\bar h_{123}-2(\bar\lambda^{2}_{1}+\bar\lambda^{2}_{3}
+\bar\lambda_{1}\bar\lambda_{3})\bar h_{131}\bar h_{133}-2(\bar\lambda^{2}_{2}+\bar\lambda^{2}_{3} \\
&+\bar\lambda_{2}\bar\lambda_{3})\bar h_{231}\bar h_{233},\\

&\bar \lambda^{3}_{1}\bar h_{1123}+\bar \lambda^{3}_{2}\bar h_{2223}+\bar \lambda^{3}_{3}\bar h_{3323}\\
=&-3\bar \lambda^{2}_{1}\bar h_{112}\bar h_{113}-3\bar \lambda^{2}_{2}\bar h_{222}\bar h_{223}-3\bar \lambda^{2}_{3}
\bar h_{332}\bar h_{333}-2(\bar \lambda^{2}_{1}+\bar \lambda^{2}_{2} \\
&+\bar \lambda_{1}\bar \lambda_{2})\bar h_{122}\bar h_{123}-2(\bar\lambda^{2}_{1}+\bar\lambda^{2}_{3}
+\bar\lambda_{1}\bar\lambda_{3})\bar h_{132}\bar h_{133}-2(\bar\lambda^{2}_{2}+\bar\lambda^{2}_{3} \\
&+\bar\lambda_{2}\bar\lambda_{3})\bar h_{232}\bar h_{233}.
\end{aligned}
\end{cases}
\end{equation}
\noindent Now we consider three scenarios.
\vskip2mm
\noindent {\bf 1. $\bar \lambda_1$, $\bar \lambda_2$ and $\bar \lambda_3$ are all equal.}

\noindent  From $\bar H=\bar \lambda_1+\bar \lambda_2+\bar \lambda_3=0$, $\bar \lambda_1=\bar \lambda_2=\bar \lambda_3=0$, we get $S=0$. It is impossible since $S\geq1$.

\vskip2mm
\noindent {\bf 2. Two of the values of $\bar \lambda_1$, $\bar \lambda_2$ and $\bar \lambda_3$  are equal.}

\noindent Without loss of generality, we assume that $\bar \lambda_1=\bar \lambda_2\neq \bar \lambda_3$.

\noindent From $\bar H=\bar \lambda_1+\bar \lambda_2+\bar \lambda_3=0$, we infer that $\bar \lambda_{1}=\bar \lambda_{2}\neq 0$ and $\bar\lambda_3\neq0$.

\noindent By \eqref{2.1-17} in the Lemma 2.3,  we obtain
$$\lim_{t\rightarrow\infty}f_{3}(p_{t})\neq 0; \ \ \ \ \ \bar H_{,k}=0 \ \ {\text for}\  k=1, 2, 3.$$

\noindent By \eqref{3.1-1} and $\bar H_{,k}=0$ for $k=1, 2, 3$, we have

$$\bar H_{,k}=\bar \lambda_{k}\lim_{t\rightarrow\infty}\langle T, e_{k}\rangle(p_{t})=0,  \ \ \lim_{t\rightarrow\infty}\langle T, e_{k}\rangle(p_{t})=0, \ \ {\text for}\ k=1, 2, 3.$$

\noindent From \eqref{3.1-3}, we have

\begin{equation}\label{3.1-11}
\begin{cases}
\begin{aligned}
&\bar h_{1111}+\bar h_{2211}+\bar h_{3311}=\bar \lambda_{1},\\
&\bar h_{1122}+\bar h_{2222}+\bar h_{3322}=\bar \lambda_{2},\\
&\bar h_{1133}+\bar h_{2233}+\bar h_{3333}=\bar \lambda_{3}.
\end{aligned}
\end{cases}
\end{equation}

\noindent From \eqref{2.1-18}, $\lim_{t\rightarrow\infty}f_{3}(p_{t})\neq 0$ and $\bar H_{,k}=0$ for $k=1, 2, 3$, we have
\begin{equation*}
\lim_{t\rightarrow\infty}\nabla_{l}\nabla_{k}f_{4}(p_{t})=0,\ \ \bar H_{,kl}=0,\ \ {\text for}\ k, l=1, 2, 3.
\end{equation*}
Then, it follows from \eqref{3.1-11} that $\bar H_{,kk}=\bar \lambda_{k}=0$ for $k=1, 2, 3$. It is a contradiction.

\vskip2mm
\noindent {\bf 3. The values of $\bar \lambda_1$, $\bar \lambda_{2}$ and $\bar \lambda_3$ are not equal to each other.}

\noindent {\bf Case 1: $\bar \lambda_1\bar \lambda_2\bar \lambda_3=0$}.

\noindent Without loss of generality, we assume that $\bar \lambda_3=0$. That is, $\bar \lambda_{1}\neq 0$, $\bar \lambda_{2}\neq 0$ and $\bar \lambda_{1}\neq\bar \lambda_{2}$.

\noindent From $\bar H=0$ and $S\neq0$, we have that $\bar \lambda_{1}=-\bar \lambda_{2}\neq 0$, $S=2\bar \lambda^{2}_{1}$ and $\lim_{t\rightarrow\infty}f_{3}(p_{t})=0$.

\noindent By \eqref{3.1-1} and \eqref{3.1-4}, we have
\begin{equation}\label{3.1-12}
\begin{cases}
\begin{aligned}
&\bar h_{111}=\bar h_{221},\ \ \bar h_{112}=\bar h_{222}, \ \ \bar h_{113}=\bar h_{223}, \\
&\lim_{t\rightarrow\infty}\langle X, e_{1}\rangle(p_{t})=\frac{2\bar h_{111}+\bar h_{133}}{\bar \lambda_{1}},\ \
 \lim_{t\rightarrow\infty}\langle X, e_{2}\rangle(p_{t})=-\frac{2\bar h_{112}+\bar h_{233}}{\bar \lambda_{1}}, \\
&\bar H_{,3}=0, \ \ \bar h_{333}=-2\bar h_{113}.
\end{aligned}
\end{cases}
\end{equation}

\noindent From \eqref{3.1-5}, \eqref{3.1-10} and \eqref{3.1-12}, we get
\begin{equation}\label{3.1-13}
\begin{cases}
\begin{aligned}
&\bar \lambda_{1}(\bar h_{1111}-\bar h_{2211})=-2\bar h^{2}_{111}-\bar h^{2}_{133}-2\bar h^{2}_{112}-2\bar h^{2}_{113}-2\bar h^{2}_{123}, \\
&\bar \lambda_{1}(\bar h_{1122}-\bar h_{2222})=-2\bar h^{2}_{112}-\bar h^{2}_{233}-2\bar h^{2}_{111}-2\bar h^{2}_{123}-2\bar h^{2}_{113}, \\
&\bar \lambda_{1}(\bar h_{1133}-\bar h_{2233})=-6\bar h^{2}_{113}-2\bar h^{2}_{123}-2\bar h^{2}_{133}-2\bar h^{2}_{233}, \\
&\bar \lambda_{1}(\bar h_{1112}-\bar h_{2212})=-4\bar h_{111}\bar h_{112}-\bar h_{133}\bar h_{233}-4\bar h_{113}\bar h_{123}, \\
&\bar \lambda_{1}(\bar h_{1113}-\bar h_{2213})=-2\bar h_{111}\bar h_{113}-2\bar h_{112}\bar h_{123}-2\bar h_{233}\bar h_{123}, \\
&\bar \lambda_{1}(\bar h_{1123}-\bar h_{2223})=-2\bar h_{112}\bar h_{113}-2\bar h_{111}\bar h_{123}-2\bar h_{133}\bar h_{123},
\end{aligned}
\end{cases}
\end{equation}
and
\begin{equation}\label{3.1-14}
\begin{cases}
\begin{aligned}
\bar \lambda^{3}_{1}(\bar h_{1111}-\bar h_{2211})
=&-6\bar \lambda^{2}_{1}\bar h^{2}_{111}-2\bar \lambda^{2}_{1}\bar h^{2}_{112}-2\bar \lambda^{2}_{1}\bar h^{2}_{113}-2\bar \lambda^{2}_{1}\bar h^{2}_{123},\\

\bar \lambda^{3}_{1}(\bar h_{1122}-\bar h_{2222})
=&-6\bar \lambda^{2}_{1}\bar h^{2}_{112}-2\bar \lambda^{2}_{1}\bar h^{2}_{111}-2\bar \lambda^{2}_{1}\bar h^{2}_{123}-2\bar \lambda^{2}_{1}\bar h^{2}_{113},\\

\bar \lambda^{3}_{1}(\bar h_{1133}-\bar h_{2233})
=&-6\bar \lambda^{2}_{1}\bar h^{2}_{113}-2\bar \lambda^{2}_{1}\bar h^{2}_{123}-2\bar \lambda^{2}_{1}\bar h^{2}_{133}-2\bar \lambda^{2}_{1}\bar h^{2}_{233},\\
\bar \lambda^{3}_{1}(\bar h_{1112}-\bar h_{2212})
=&-8\bar \lambda^{2}_{1}\bar h_{111}\bar h_{112}-4\bar \lambda^{2}_{1}\bar h_{113}\bar h_{123},\\

\bar \lambda^{3}_{1}(\bar h_{1113}-\bar h_{2213})
=&-6\bar \lambda^{2}_{1}\bar h_{111}\bar h_{113}-2\bar \lambda^{2}_{1}\bar h_{112}\bar h_{123}-2\bar \lambda^{2}_{1}\bar h_{113}\bar h_{133}\\
 &-2\bar \lambda^{2}_{1}\bar h_{233}\bar h_{123},\\

\bar \lambda^{3}_{1}(\bar h_{1123}-\bar h_{2223})
=&-6\bar \lambda^{2}_{1}\bar h_{112}\bar h_{113}-2\bar \lambda^{2}_{1}\bar h_{111}\bar h_{123}-2\bar \lambda^{2}_{1}\bar h_{133}\bar h_{123}\\
 &-2\bar \lambda^{2}_{1}\bar h_{113}\bar h_{233}.\\
\end{aligned}
\end{cases}
\end{equation}

\noindent From \eqref{3.1-13} and \eqref{3.1-14}, we have
\begin{equation}\label{3.1-15}
\begin{aligned}
& 4\bar h^{2}_{111}=\bar h^{2}_{133}, \ \ \ \  4\bar h^{2}_{112}=\bar h^{2}_{233},\ \ \ \ 4\bar h_{111}\bar h_{112}=\bar h_{133}\bar h_{233},\\
&\bar h_{113}(2\bar h_{111}+\bar h_{133})=0, \ \ \ \ \bar h_{113}(2\bar h_{112}+\bar h_{233})=0.
\end{aligned}
\end{equation}

\noindent Supposing $\bar h_{113} \neq 0$. From \eqref{3.1-1} and \eqref{3.1-15}, we know
\begin{equation}\label{3.1-16}
\begin{aligned}
&2\bar h_{111}+\bar h_{133}=0, \ \ 2\bar h_{112}+\bar h_{233}=0, \\
&\bar H_{,1}=2\bar h_{111}+\bar h_{133}=\bar \lambda_{1}\lim_{t\rightarrow\infty}\langle X, e_{1}\rangle(p_{t})=0, \ \ \lim_{t\rightarrow\infty}\langle X, e_{1}\rangle(p_{t})=0,\\
&\bar H_{,2}=2\bar h_{112}+\bar h_{233}=\bar \lambda_{2}\lim_{t\rightarrow\infty}\langle X, e_{2}\rangle(p_{t})=0, \ \ \lim_{t\rightarrow\infty}\langle X, e_{2}\rangle(p_{t})=0,
\end{aligned}
\end{equation}
and then, $\bar H_{,k}=0$ for $k=1, 2, 3$.

\noindent From \eqref{3.1-7}, we know
\begin{equation}\label{3.1-17}
\lim_{t\rightarrow\infty}\nabla_{3}f_{3}(p_{t})=6\bar \lambda^{2}_{1}\bar h_{113}=3S\bar h_{113}.
\end{equation}

\noindent By \eqref{3.1-3} and \eqref{3.1-12}, we have
\begin{equation}\label{3.1-18}
\bar H_{,33}=\bar h_{1133}+\bar h_{2233}+\bar h_{3333}=-2\bar h_{113}\lim_{t\rightarrow\infty}\langle X, e_{3}\rangle(p_{t}).
\end{equation}

\noindent From \eqref{2.1-19}, $\lim_{t\rightarrow\infty}f_{3}(p_{t})=0$ and $\bar H_{,k}=0$ for $k=1, 2, 3$, we know
\begin{equation}\label{3.1-19}
\frac{4}{3}\bar H_{,kl}\lim_{t\rightarrow\infty}\nabla_{m}f_{3}(p_{t})+\frac{4}{3}\bar H_{,km}\lim_{t\rightarrow\infty}\nabla_{l}f_{3}(p_{t})
+\frac{4}{3}\bar H_{,lm}\lim_{t\rightarrow\infty}\nabla_{k}f_{3}(p_{t})=0,
\end{equation}
where $k, l, m=1, 2, 3$.

\noindent Choosing $k=l=m=3$ in \eqref{3.1-19}, and by \eqref{3.1-17} and \eqref{3.1-18}, we obtain

\begin{equation*}
0= 4\bar H_{,33}\lim_{t\rightarrow\infty}\nabla_{3}f_{3}(p_{t})=-24S\bar h^{2}_{113}\lim_{t\rightarrow\infty}\langle X, e_{3}\rangle(p_{t}).
\end{equation*}
Therefore,
\begin{equation}\label{3.1-20}
\lim_{t\rightarrow\infty}\langle X, e_{3}\rangle(p_{t})=0, \ \ \bar H_{,33}=0.
\end{equation}

\noindent From $\lim_{t\rightarrow\infty}f_{3}(p_{t})=0$, $\bar H_{,k}=0$ for $k=1, 2, 3$ and choosing $k=l=m=n=3$ in \eqref{2.1-20}, we know
\begin{equation*}
\frac{16}{3}\lim_{t\rightarrow\infty}\nabla_{3}f_{3}(p_{t})\bar H_{,333}=0,
\end{equation*}
and then, $$\bar H_{,333}=0.$$

\noindent Since
\begin{equation*}
\begin{aligned}
\nabla_{l}\nabla_{k}\nabla_{j}H
 =&\sum_{i}h_{ijkl}\langle X,e_{i}\rangle+2h_{jkl}-\sum_{i}(h_{ijk}h_{il}+h_{ijl}h_{ik}+h_{ikl}h_{ij})H \\
  &-\sum_{i,p}h_{ij}h_{ik}h_{lp}\langle X,e_{p}\rangle,
\end{aligned}
\end{equation*}
we have
\begin{equation*}
\bar H_{,333}=2\bar h_{333}=0, \ \ \bar h_{113}=-\frac{1}{2}\bar h_{333}=0.
\end{equation*}
It contradicts the hypothesis. We have $\bar h_{113}=0$.

\noindent {\bf Subcase 1.1: $2\bar h_{111}+\bar h_{133}=0$.}

\noindent Supposing $\bar h_{111} \neq 0$. From $\bar h_{113}=0$, \eqref{3.1-12} and \eqref{3.1-15}, we have
\begin{equation*}
\bar h_{223}=\bar h_{333}=0, \ \ 2\bar h_{112}+\bar h_{233}=0.
\end{equation*}
Then,
\begin{equation*}
\lim_{t\rightarrow\infty}\langle T, e_{1}\rangle(p_{t})=\lim_{t\rightarrow\infty}\langle T, e_{2}\rangle(p_{t})=0,\ \ \bar H_{,k}=0, \ \ \text{for }\ k=1, 2, 3.
\end{equation*}

\noindent From \eqref{3.1-3} and \eqref{3.1-7}, we have
\begin{equation*}
\bar H_{,11}=\bar \lambda_{1}, \ \ \lim_{t\rightarrow\infty}\nabla_{1}f_{3}(p_{t})=6\bar \lambda^{2}_{1}\bar h_{111}=3S\bar h_{111}.
\end{equation*}

\noindent Choosing $k=l=m=1$ in \eqref{3.1-19}, we obtain
\begin{equation*}
0=4\lim_{t\rightarrow\infty}\nabla_{1}f_{3}(p_{t})\bar H_{,11}=12\bar \lambda_{1}S\bar h_{111},
\end{equation*}
and then,
\begin{equation*}
\bar h_{111}=0,
\end{equation*}
It contradicts the hypothesis. We have that $\bar h_{111}=\bar h_{133}=0$.

\noindent From \eqref{3.1-15}, we have that either $2\bar h_{112}+\bar h_{233}=0$ or $2\bar h_{112}-\bar h_{233}=0$.

\noindent If $2\bar h_{112}+\bar h_{233}=0$, from $2\bar h_{111}+\bar h_{133}=0$ and \eqref{3.1-1}, we know
\begin{equation*}
\begin{aligned}
&\bar H_{,1}=\bar \lambda_{1}\lim_{t\rightarrow\infty} \langle X, e_{2} \rangle(p_{t})=0, \ \ \lim_{t\rightarrow\infty}\langle X, e_{1}\rangle(p_{t})=0,\\
&\bar H_{,2}=\bar \lambda_{2}\lim_{t\rightarrow\infty} \langle X, e_{2} \rangle(p_{t})=0, \ \ \lim_{t\rightarrow\infty}\langle X, e_{2}\rangle(p_{t})=0,
\end{aligned}
\end{equation*}
and then, $\bar H_{,k}=0$ for $k=1, 2, 3$.

\noindent From \eqref{3.1-3} and \eqref{3.1-7}, we obtain
\begin{equation*}
\bar H_{,22}=-\bar \lambda_{1}, \ \ \lim_{t\rightarrow\infty}\nabla_{2}f_{3}(p_{t})=6\bar \lambda^{2}_{1}\bar h_{112}=3S\bar h_{112}.
\end{equation*}

\noindent Choosing $k=l=m=2$ in \eqref{3.1-19}, we obtain
\begin{equation*}
0=4\lim_{t\rightarrow\infty}\nabla_{2}f_{3}(p_{t})\bar H_{,22}
 =-12\bar \lambda_{1}S\bar h_{112},
\end{equation*}
and then,
\begin{equation*}
\bar h_{112}=\bar h_{233}=0, \ \ \lim_{t\rightarrow\infty}\nabla_{2}f_{3}(p_{t})=0.
\end{equation*}
\noindent From \eqref{2.1-20}, $\lim_{t\rightarrow\infty}f_{3}(p_{t})=0$ and $\bar H_{,k}=0$ for $k=1, 2, 3$, we know
\begin{equation}\label{3.1-21}
\begin{aligned}
&\frac{4}{3}\lim_{t\rightarrow\infty}\nabla_{n}f_{3}(p_{t})\bar H_{,klm}+\frac{4}{3}\lim_{t\rightarrow\infty}\nabla_{m}f_{3}(p_{t})\bar H_{,kln}
+\frac{4}{3}\lim_{t\rightarrow\infty}\nabla_{n}\nabla_{m}f_{3}(p_{t})\bar H_{,kl} \\
&-2S\bar H_{,mn}\bar H_{,kl}+\frac{4}{3}\bar H_{,mn}\lim_{t\rightarrow\infty}\nabla_{l}\nabla_{k}f_{3}(p_{t})
+\frac{4}{3}\lim_{t\rightarrow\infty}\nabla_{l}f_{3}(p_{t})\bar H_{,kmn}\\
&+\frac{4}{3}\lim_{t\rightarrow\infty}\nabla_{n}\nabla_{l}f_{3}(p_{t})\bar H_{,km}+\frac{4}{3}\lim_{t\rightarrow\infty}\nabla_{m}\nabla_{l}f_{3}(p_{t})\bar H_{,kn}+\frac{4}{3}\bar H_{,ln}\lim_{t\rightarrow\infty}\nabla_{m}\nabla_{k}f_{3}(p_{t}) \\
&+\frac{4}{3}\bar H_{,lmn}\lim_{t\rightarrow\infty}\nabla_{k}f_{3}(p_{t})+\frac{4}{3}\bar H_{,lm}\lim_{t\rightarrow\infty}\nabla_{n}\nabla_{k}f_{3}(p_{t})-2S\bar H_{,km}\bar H_{,ln}\\
&-2S\bar H_{,kn}\bar H_{,lm}=0,
\end{aligned}
\end{equation}
where $k, l, m, n=1, 2, 3$.

\noindent From $\bar h_{111}=\bar h_{112}=\bar h_{113}=0$, \eqref{3.1-7} and  \eqref{3.1-12}, we have
\begin{equation}\label{3.1-22}
\lim_{t\rightarrow\infty}\nabla_{k}f_{3}(p_{t})=6\bar \lambda^{2}_{k}\bar h_{11k}=0, \ \ k=1, 2, 3.
\end{equation}

\noindent From \eqref{3.1-21} and \eqref{3.1-22}, we have
\begin{equation}\label{3.1-23}
\begin{cases}
\begin{aligned}
&\frac{4}{3}\lim_{t\rightarrow\infty}\nabla_{n}\nabla_{m}f_{3}(p_{t})\bar H_{,kl}+\frac{4}{3}\bar H_{,mn}\lim_{t\rightarrow\infty}\nabla_{l}\nabla_{k}f_{3}(p_{t})+\frac{4}{3}\lim_{t\rightarrow\infty}\nabla_{n}\nabla_{l}f_{3}(p_{t})\bar H_{,km} \\
&+\frac{4}{3}\lim_{t\rightarrow\infty}\nabla_{m}\nabla_{l}f_{3}(p_{t})\bar H_{,kn}+\frac{4}{3}\bar H_{,ln}\lim_{t\rightarrow\infty}\nabla_{m}\nabla_{k}f_{3}(p_{t})+\frac{4}{3}\bar H_{,lm}\lim_{t\rightarrow\infty}\nabla_{n}\nabla_{k}f_{3}(p_{t}) \\
&-2S\bar H_{,mn}\bar H_{,kl}-2S\bar H_{,km}\bar H_{,ln}-2S\bar H_{,kn}\bar H_{,lm}=0,
\ \ \text{for } \ k, l, m, n=1, 2, 3.
\end{aligned}
\end{cases}
\end{equation}

\noindent Using \eqref{3.1-3} and \eqref{3.1-13}, we have that
\begin{equation}\label{3.1-24}
\begin{cases}
\begin{aligned}
&\bar H_{,11}=\bar h_{1111}+\bar h_{2211}+\bar h_{3311}=\bar \lambda_{1},\\
&\bar H_{,22}=\bar h_{1122}+\bar h_{2222}+\bar h_{3322}=-\bar \lambda_{1},\\
&\bar H_{,33}=\bar h_{1133}+\bar h_{2233}+\bar h_{3333}=0,\\
&\bar H_{,12}=\bar h_{1112}+\bar h_{2212}+\bar h_{3312}=\bar h_{123}\lim_{t\rightarrow\infty}\langle X,e_{3} \rangle(p_{t}),\\
&\bar H_{,13}=\bar h_{1113}+\bar h_{2213}+\bar h_{3313}=0,\\
&\bar H_{,23}=\bar h_{1123}+\bar h_{2223}+\bar h_{3323}=0,
\end{aligned}
\end{cases}
\end{equation}
and
\begin{equation}\label{3.1-25}
\begin{cases}
\begin{aligned}
&\bar \lambda_{1}(\bar h_{1111}-\bar h_{2211})=-2\bar h^{2}_{123},\\
&\bar \lambda_{1}(\bar h_{1122}-\bar h_{2222})=-2\bar h^{2}_{123},\\
&\bar \lambda_{1}(\bar h_{1133}-\bar h_{2233})=-2\bar h^{2}_{123},\\
&\bar \lambda_{1}(\bar h_{1112}-\bar h_{2212})=0,\\
&\bar \lambda_{1}(\bar h_{1113}-\bar h_{2213})=0,\\
&\bar \lambda_{1}(\bar h_{1123}-\bar h_{2223})=0.
\end{aligned}
\end{cases}
\end{equation}

\noindent Choosing $k=l=m=n=1$; $k=l=m=n=2$; $k=l=m=n=3$ and $k=l=1,\ \ m=n=2$ in \eqref{3.1-23}, respectively, we obtain
\begin{equation*}
\begin{cases}
\begin{aligned}
&\bar H_{,11}\lim_{t\rightarrow\infty}\nabla_{1}\nabla_{1}f_{3}(p_{t})=\frac{3}{4}S(\bar H_{,11})^{2},\\
&\bar H_{,22}\lim_{t\rightarrow\infty}\nabla_{2}\nabla_{2}f_{3}(p_{t})=\frac{3}{4}S(\bar H_{,22})^{2},\\
&\bar H_{,33}\lim_{t\rightarrow\infty}\nabla_{3}\nabla_{3}f_{3}(p_{t})=\frac{3}{4}S(\bar H_{,33})^{2},\\
&\frac{4}{3}\lim_{t\rightarrow\infty}\nabla_{1}\nabla_{1}f_{3}(p_{t})\bar H_{,22}+\frac{4}{3}\lim_{t\rightarrow\infty}\nabla_{2}\nabla_{2}f_{3}(p_{t})\bar H_{,11}\\
&+\frac{16}{3}\lim_{t\rightarrow\infty}\nabla_{2}\nabla_{1}f_{3}(p_{t})\bar H_{,12} -2S\bar H_{,11}\bar H_{,22}-4S(\bar H_{,12})^{2}=0.
\end{aligned}
\end{cases}
\end{equation*}
Then,
\begin{equation}\label{3.1-26}
\begin{cases}
\begin{aligned}
&\lim_{t\rightarrow\infty}\nabla_{1}\nabla_{1}f_{3}(p_{t})=\frac{3}{4}S\bar H_{,11},\\
&\lim_{t\rightarrow\infty}\nabla_{2}\nabla_{2}f_{3}(p_{t})=\frac{3}{4}S\bar H_{,22},\\
&\lim_{t\rightarrow\infty}\nabla_{3}\nabla_{3}f_{3}(p_{t})=\frac{3}{4}S\bar H_{,33},\\
&\frac{4}{3}\lim_{t\rightarrow\infty}\nabla_{1}\nabla_{1}f_{3}(p_{t})\bar H_{,22}+\frac{4}{3}\lim_{t\rightarrow\infty}\nabla_{2}\nabla_{2}f_{3}(p_{t})\bar H_{,11}\\
&+\frac{16}{3}\lim_{t\rightarrow\infty}\nabla_{2}\nabla_{1}f_{3}(p_{t})\bar H_{,12} -2S\bar H_{,11}\bar H_{,22}-4S(\bar H_{,12})^{2}=0.
\end{aligned}
\end{cases}
\end{equation}

\noindent From \eqref{3.1-8}, \eqref{3.1-24} and \eqref{3.1-26}, we know
\begin{equation*}
\begin{cases}
\begin{aligned}
&\lim_{t\rightarrow\infty}\nabla_{1}\nabla_{1}f_{3}(p_{t})=3\bar \lambda^{2}_{1}(\bar h_{1111}+\bar h_{2211})-6\bar \lambda_{1}\bar h^{2}_{123}
=\frac{3}{4}S\bar \lambda_{1},\\
&\lim_{t\rightarrow\infty}\nabla_{2}\nabla_{2}f_{3}(p_{t})=3\bar \lambda^{2}_{1}(\bar h_{1122}+\bar h_{2222})+6\bar \lambda_{1}\bar h^{2}_{123}
=-\frac{3}{4}S\bar \lambda_{1},\\
&\lim_{t\rightarrow\infty}\nabla_{3}\nabla_{3}f_{3}(p_{t})=3\bar \lambda^{2}_{1}(\bar h_{1133}+\bar h_{2233})=0,\\
&\frac{4}{3}\lim_{t\rightarrow\infty}\nabla_{1}\nabla_{1}f_{3}(p_{t})\bar H_{,22}+\frac{4}{3}\lim_{t\rightarrow\infty}\nabla_{2}\nabla_{2}f_{3}(p_{t})\bar H_{,11}+\frac{16}{3}\lim_{t\rightarrow\infty}\nabla_{2}\nabla_{1}f_{3}(p_{t})\bar H_{,12} \\
&-2S\bar H_{,11}\bar H_{,22}-4S(\bar H_{,12})^{2}=0.
\end{aligned}
\end{cases}
\end{equation*}
Then,
\begin{equation}\label{3.1-27}
\begin{cases}
\begin{aligned}
&\bar \lambda_{1}(\bar h_{1111}+\bar h_{2211})=\frac{1}{4}S+2\bar h^{2}_{123},\\
&\bar \lambda_{1}(\bar h_{1122}+\bar h_{2222})=-\frac{1}{4}S-2\bar h^{2}_{123},\\
&\bar h_{1133}=-\bar h_{2233},\\
&\frac{16}{3}\lim_{t\rightarrow\infty}\nabla_{2}\nabla_{1}f_{3}(p_{t})\bar H_{,12}-4S(\bar H_{,12})^{2}=0.
\end{aligned}
\end{cases}
\end{equation}

\noindent If $\bar h_{123}=0$, we have $\bar h_{ijk}=0$ for  $i, j, k=1, 2, 3$.
From \eqref{2.1-16} and \eqref{3.1-24},
we have $\bar h_{ijkl}=0$ for $i, j, k, l=1, 2, 3$, and then $\bar \lambda_{1}=0$, this contradicts the hypothesis.
Hence, we get $\bar h_{123}\neq 0$.

\noindent From \eqref{2.1-15} in Lemma 2.1, we have
\begin{equation*}
\sum_{i,j,k}h_{ijk}^2+(1-S)S=0, \ \ \sum_{i,j,k}h_{ijk}h_{ijkl}=0,
\end{equation*}
for $l=1,2,3$.
Then,
\begin{equation}\label{3.1-28}
\begin{aligned}
&\bar h_{123}^2=\frac{1}{6}(S-1)S, \\
&6\bar h_{123}\bar h_{1231}=0, \ \ \bar h_{1231}=\bar h_{1123}=0,\\
&6\bar h_{123}\bar h_{1232}=0, \ \ \bar h_{1232}=\bar h_{2213}=0,\\
&6\bar h_{123}\bar h_{1233}=0, \ \ \bar h_{1233}=\bar h_{3312}=0.
\end{aligned}
\end{equation}

\noindent From \eqref{3.1-24}, \eqref{3.1-25} and \eqref{3.1-28}, we know

\begin{equation}\label{3.1-29}
\begin{cases}
\begin{aligned}
&\bar h_{1111}-\bar h_{2211}=-\frac{2\bar h^{2}_{123}}{\bar \lambda_{1}}, \ \ \bar h_{1122}-\bar h_{2222}=-\frac{2\bar h^{2}_{123}}{\bar \lambda_{1}},\\
&\bar h_{1133}-\bar h_{2233}=-\frac{2\bar h^{2}_{123}}{\bar \lambda_{1}}, \ \  \bar h_{1112}=\bar h_{2212}, \ \ \bar h_{3312}=0,\\
&\bar h_{1113}=\bar h_{2213}=\bar h_{3313}=0, \ \ \bar h_{1123}=\bar h_{2223}=\bar h_{3323}=0.
\end{aligned}
\end{cases}
\end{equation}
By \eqref{3.1-8} and \eqref{3.1-29}, one has
\begin{equation}\label{3.1-30}
 \frac{1}{3}\lim_{t\rightarrow\infty}\nabla_{2}\nabla_{1}f_{3}(p_t)=2\bar \lambda^{2}_{1}\bar h_{1112}=S\bar h_{1112}, \ \ \bar H_{,12}=2\bar h_{1112}.
\end{equation}
From the fourth equation in \eqref{3.1-27} and \eqref{3.1-30}, we obtain
\begin{equation}\label{3.1-31}
\bar h_{1112}=0.
\end{equation}

\noindent From \eqref{3.1-24}, \eqref{3.1-27} and \eqref{3.1-29}, we get

\begin{equation}\label{3.1-32}
\begin{cases}
\begin{aligned}
&\bar h_{1111}=\frac{S}{8\bar \lambda_{1}}, \ \ h_{2211}=\frac{S}{8\bar \lambda_{1}}+\frac{2\bar h^{2}_{123}}{\bar \lambda_{1}}, \ \
\bar h_{3311}=\frac{S}{4\bar \lambda_{1}}-\frac{2\bar h^{2}_{123}}{\bar \lambda_{1}},\\
&\bar h_{1122}=-(\frac{S}{8\bar \lambda_{1}}+\frac{2\bar h^{2}_{123}}{\bar \lambda_{1}}), \ \ \bar h_{2222}=-\frac{S}{8\bar \lambda_{1}}, \ \
\bar h_{3322}=-\frac{S}{4\bar \lambda_{1}}+\frac{2\bar h^{2}_{123}}{\bar \lambda_{1}},\\
&\bar h_{1133}=-\frac{\bar h^{2}_{123}}{\bar \lambda_{1}}, \ \ \bar h_{2233}=\frac{\bar h^{2}_{123}}{\bar \lambda_{1}}, \ \ \bar h_{3333}=0.
\end{aligned}
\end{cases}
\end{equation}

\noindent From Ricci identities \eqref{3.1-6}, \eqref{3.1-28} and \eqref{3.1-32}, we have

\begin{equation}\label{3.1-33}
\bar h_{1133}-\bar h_{3311}=-\frac{S}{4\bar \lambda_{1}}+\frac{\bar h^{2}_{123}}{\bar \lambda_{1}}=\bar \lambda_{1}\bar \lambda_{3}(\bar \lambda_{1}-\bar \lambda_{3})=0.
\end{equation}
Then, $$S=\frac{5}{2}.$$

\noindent From \eqref{2.1-16}, \eqref{3.1-29}, \eqref{3.1-31} and \eqref{3.1-32}, we have
\begin{equation*}
\begin{aligned}
0=&\frac{1}{2}\lim_{t\rightarrow\infty}\mathcal{L}\sum_{i, j,k}(h_{ijk})^{2}(p_{t})\\
 =&\sum_{i,j,k,l}(\bar h_{ijkl})^{2}+(2-S)\sum_{i,j,k}(\bar h_{ijk})^{2}+6\sum_{i,j,k,l,p}\bar h_{ijk}\bar h_{il}\bar h_{jp}\bar h_{klp}\\
 &-3\sum_{i,j,k,l,p}\bar h_{ijk}\bar h_{ijl}\bar h_{kp}\bar h_{lp}\\
 =&\bar h_{1111}^2+\bar h_{2222}^2+3\bar h_{1122}^2+3\bar h_{2211}^2+3\bar h_{1133}^2+3\bar h_{3311}^2+3\bar h_{2233}^2+3\bar h_{3322}^2\\
&+6(2-S)\bar h_{123}^2-6S\bar h_{123}^2-6S\bar h_{123}^2\\
=&S+\frac{108}{S}\bar h_{123}^4+6\bar h_{123}^2-18 S\bar h_{123}^2\\
 =&S(3-2S),
\end{aligned}
\end{equation*}
then $S=\frac{3}{2}$ which contradicts to $S=\frac{5}{2}$,  where we use $\bar h_{123}^2=\frac{1}{6}S(S-1)$.

\noindent If $2\bar h_{112}-\bar h_{233}=0$, from \eqref{3.1-7} and \eqref{3.1-12}, we know
\begin{equation}\label{3.1-34}
\begin{aligned}
&\lim_{t\rightarrow\infty}\nabla_{2}f_{3}(p_{t})=6\bar \lambda^{2}_{1}\bar h_{112}=3S\bar h_{112},\\
&\bar H_{,2}=4\bar h_{112}, \ \ \lim_{t\rightarrow\infty}\langle X, e_{2}\rangle(p_{t})=-\frac{4\bar h_{112}}{\bar \lambda_{1}}.
\end{aligned}
\end{equation}

\noindent By \eqref{3.1-3} and \eqref{3.1-34}, we have
\begin{equation}\label{3.1-35}
\begin{cases}
\begin{aligned}
&\bar H_{,11}=\bar h_{1111}+\bar h_{2211}+\bar h_{3311}=-\frac{4\bar h^{2}_{112}}{\bar \lambda_{1}}+\bar \lambda_{1},\\
&\bar H_{,22}=\bar h_{1122}+\bar h_{2222}+\bar h_{3322}=-\frac{4\bar h^{2}_{112}}{\bar \lambda_{1}}-\bar \lambda_{1},\\
&\bar H_{,33}=\bar h_{1133}+\bar h_{2233}+\bar h_{3333}=-\frac{8\bar h^{2}_{112}}{\bar \lambda_{1}}.
\end{aligned}
\end{cases}
\end{equation}

\noindent From \eqref{3.1-8} and \eqref{3.1-13}, we obtain
\begin{equation}\label{3.1-36}
\begin{cases}
\begin{aligned}
&\frac{1}{3}\lim_{t\rightarrow\infty}\nabla_{1}\nabla_{1}f_{3}(p_{t})=\bar \lambda^{2}_{1}(\bar h_{1111}+\bar h_{2211})-2\bar \lambda_{1}\bar h^{2}_{123},\\
&\frac{1}{3}\lim_{t\rightarrow\infty}\nabla_{2}\nabla_{2}f_{3}(p_{t})=\bar \lambda^{2}_{1}(\bar h_{1122}+\bar h_{2222})+2\bar \lambda_{1}\bar h^{2}_{123},\\
&\frac{1}{3}\lim_{t\rightarrow\infty}\nabla_{3}\nabla_{3}f_{3}(p_{t})=\bar \lambda^{2}_{1}(\bar h_{1133}+\bar h_{2233})-8\bar \lambda_{1}\bar h^{2}_{112},
\end{aligned}
\end{cases}
\end{equation}
and
\begin{equation}\label{3.1-37}
\begin{cases}
\begin{aligned}
&\bar \lambda_{1}(\bar h_{1111}-\bar h_{2211})=-2\bar h^{2}_{112}-2\bar h^{2}_{123},\\
&\bar \lambda_{1}(\bar h_{1122}-\bar h_{2222})=-6\bar h^{2}_{112}-2\bar h^{2}_{123},\\
&\bar \lambda_{1}(\bar h_{1133}-\bar h_{2233})=-8\bar h^{2}_{112}-2\bar h^{2}_{123}.
\end{aligned}
\end{cases}
\end{equation}

\noindent From  $\lim_{t\rightarrow\infty}f_{3}(p_{t})=0$ and \eqref{2.1-19} in the Lemma \ref{lemma 2.3}, we know
\begin{equation}\label{3.1-38}
\begin{aligned}
&\frac{4}{3}\nabla_{m}f_{3}H_{,kl}-2SH_{,m}H_{,kl}+\frac{4}{3}H_{,m}\nabla_{l}\nabla_{k}f_{3}+\frac{4}{3}\nabla_{l}f_{3} H_{,km}+\frac{4}{3}\nabla_{m}\nabla_{l}f_{3} H_{,k} \\
&+\frac{4}{3}H_{,l}\nabla_{m}\nabla_{k}f_{3}+\frac{4}{3}H_{,lm}\nabla_{k}f_{3}-2SH_{,km}H_{,l}-2SH_{,k}H_{,lm}=0.
\end{aligned}
\end{equation}

\noindent Choosing $k=l=1,\ m=2$; $k=l=m=2$ and $k=l=3,\ m=2$ in \eqref{3.1-38}, respectively,
we obtain
\begin{equation*}
\begin{cases}
\begin{aligned}
&\frac{4}{3}\lim_{t\rightarrow\infty}\nabla_{2}f_{3}(p_{t})\bar H_{,11}+\frac{4}{3}\bar H_{,2}\lim_{t\rightarrow\infty}\nabla_{1}\nabla_{1}f_{3}(p_{t})-2S\bar H_{,2}\bar H_{,11}=0,\\
&3\biggl(\frac{4}{3}\lim_{t\rightarrow\infty}\nabla_{2}f_{3}(p_{t})\bar H_{,22}+\frac{4}{3}\bar H_{,2}\lim_{t\rightarrow\infty}\nabla_{2}\nabla_{2}f_{3}(p_{t})-2S\bar H_{,2}\bar H_{,22}\biggl)=0,\\
&\frac{4}{3}\lim_{t\rightarrow\infty}\nabla_{2}f_{3}(p_{t})\bar H_{,33}+\frac{4}{3}\bar H_{,2}\lim_{t\rightarrow\infty}\nabla_{3}\nabla_{3}f_{3}(p_{t})-2S\bar H_{,2}\bar H_{,33}=0,
\end{aligned}
\end{cases}
\end{equation*}
where $\bar H_{,1}=\bar H_{,3}=0$ and $\lim_{t\rightarrow\infty}\nabla_{1}f_{3}(p_{t})=\lim_{t\rightarrow\infty}\nabla_{3}f_{3}(p_{t})=0$.

\noindent Then, it follow that
\begin{equation}\label{3.1-39}
\begin{cases}
\begin{aligned}
&\bar h_{112}(4\lim_{t\rightarrow\infty}\nabla_{1}\nabla_{1}f_{3}(p_{t})-3S\bar H_{,11})=0,\\
&\bar h_{112}(4\lim_{t\rightarrow\infty}\nabla_{2}\nabla_{2}f_{3}(p_{t})-3S\bar H_{,22})=0,\\
&\bar h_{112}(4\lim_{t\rightarrow\infty}\nabla_{3}\nabla_{3}f_{3}(p_{t})-3S\bar H_{,33})=0.
\end{aligned}
\end{cases}
\end{equation}
When $\bar h_{112}=0$, this case is the same as that case of $2\bar h_{112}+\bar h_{233}=0$.

\noindent When $\bar h_{112} \neq 0$, from \eqref{3.1-39}, we have
\begin{equation}\label{3.1-40}
\begin{cases}
\begin{aligned}
&\frac{1}{3}\lim_{t\rightarrow\infty}\nabla_{1}\nabla_{1}f_{3}(p_{t})=\frac{1}{4}S\bar H_{,11},\\
&\frac{1}{3}\lim_{t\rightarrow\infty}\nabla_{2}\nabla_{2}f_{3}(p_{t})=\frac{1}{4}S\bar H_{,22},\\
&\frac{1}{3}\lim_{t\rightarrow\infty}\nabla_{3}\nabla_{3}f_{3}(p_{t})=\frac{1}{4}S\bar H_{,33}.
\end{aligned}
\end{cases}
\end{equation}

\noindent From \eqref{3.1-35}, \eqref{3.1-36} and \eqref{3.1-40}, we know

\begin{equation}\label{3.1-41}
\begin{cases}
\begin{aligned}
&\bar h_{1111}+\bar h_{2211}=-\frac{2(\bar h^{2}_{112}-\bar h^{2}_{123})}{\bar \lambda_{1}}+\frac{1}{2}\bar \lambda_{1},\\
&\bar h_{1122}+\bar h_{2222}=-\frac{2(\bar h^{2}_{112}+\bar h^{2}_{123})}{\bar \lambda_{1}}-\frac{1}{2}\bar \lambda_{1},\\
&\bar h_{1133}+\bar h_{2233}=\frac{4\bar h^{2}_{112}}{\bar \lambda_{1}}.
\end{aligned}
\end{cases}
\end{equation}

\noindent From \eqref{3.1-35}, \eqref{3.1-37} and \eqref{3.1-41}, we have

\begin{equation}\label{3.1-42}
\begin{cases}
\begin{aligned}
&\bar h_{3311}=\bar H_{,11}-(\bar h_{1111}+\bar h_{2211})=-\frac{2(\bar h^{2}_{112}+\bar h^{2}_{123})}{\bar \lambda_{1}}+\frac{1}{2}\bar \lambda_{1}, \\
&\bar h_{3322}=\bar H_{,22}-(\bar h_{1122}+\bar h_{2222})=-\frac{2(\bar h^{2}_{112}-\bar h^{2}_{123})}{\bar \lambda_{1}}-\frac{1}{2}\bar \lambda_{1}, \\
&\bar h_{1133}=-\frac{2\bar h^{2}_{112}+\bar h^{2}_{123}}{\bar \lambda_{1}}, \ \ \bar h_{2233}=\frac{6\bar h^{2}_{112}+\bar h^{2}_{123}}{\bar \lambda_{1}}.
\end{aligned}
\end{cases}
\end{equation}

\noindent By \eqref{3.1-42}, $\bar h_{1133}=\bar h_{3311}$ and $\bar h_{2233}=\bar h_{3322}$, we have
\begin{equation*}
\bar h^{2}_{123}=\frac{1}{2}\bar \lambda^{2}_{1}, \ \ \bar h^{2}_{112}=0.
\end{equation*}
This contradicts the hypothesis.

\noindent {\bf Subcase 1.2: $2\bar h_{111}+\bar h_{133}\neq0$.}

\noindent
From \eqref{3.1-12} and \eqref{3.1-15}, we know that
\begin{equation}\label{3.1-43}
\begin{cases}
\begin{aligned}
&\bar h_{113}=\bar h_{223}=\bar h_{333}=0, \ \ \lim_{t\rightarrow\infty}\nabla_{3}f_{3}(p_{t})=3S\bar h_{113}=0,\\
&2\bar h_{111}-\bar h_{133}=0, \ \ \bar H_{,1}=4\bar h_{111}, \\
&\lim_{t\rightarrow\infty}\langle T, e_{1}\rangle(p_{t})=\frac{4\bar h_{111}}{\bar \lambda_{1}}, \ \ \lim_{t\rightarrow\infty}\nabla_{1}f_{3}(p_{t})=3S\bar h_{111}.
\end{aligned}
\end{cases}
\end{equation}

\noindent If $\bar h_{111}=0$, we have that $\bar h_{133}=2\bar h_{111}=0$ which is impossible since $2\bar h_{111}+\bar h_{133}\neq 0$.
Hence, $\bar h_{111} \neq 0$.

\noindent
From $\bar h_{111} \neq 0$, \eqref{3.1-12} and \eqref{3.1-15}, we have
\begin{equation}\label{3.1-44}
\begin{cases}
\begin{aligned}
&2\bar h_{112}-\bar h_{233}=0, \ \ \bar H_{,2}=4\bar h_{112}, \\
&\lim_{t\rightarrow\infty}\langle T, e_{2}\rangle(p_{t})=-\frac{4\bar h_{112}}{\bar \lambda_{1}}, \ \ \lim_{t\rightarrow\infty}\nabla_{2}f_{3}(p_{t})=3S\bar h_{112}.
\end{aligned}
\end{cases}
\end{equation}

\noindent From \eqref{3.1-3}, \eqref{3.1-43} and \eqref{3.1-44}, we know
\begin{equation}\label{3.1-45}
\begin{cases}
\begin{aligned}
&\bar H_{,11}=\bar h_{1111}+\bar h_{2211}+\bar h_{3311}=\frac{4(\bar h^{2}_{111}-\bar h^{2}_{112})}{\bar \lambda_{1}}+\bar \lambda_{1},\\
&\bar H_{,22}=\bar h_{1122}+\bar h_{2222}+\bar h_{3322}=\frac{4(\bar h^{2}_{111}-\bar h^{2}_{112})}{\bar \lambda_{1}}-\bar \lambda_{1},\\
&\bar H_{,33}=\bar h_{1133}+\bar h_{2233}+\bar h_{3333}=\frac{8(\bar h^{2}_{111}-\bar h^{2}_{112})}{\bar \lambda_{1}},\\
&\bar H_{,12}=\bar h_{1112}+\bar h_{2212}+\bar h_{3312}=\bar h_{123}\lim_{t\rightarrow\infty}\langle X, e_{3}\rangle(p_{t}),\\
&\bar H_{,13}=\bar h_{1113}+\bar h_{2213}+\bar h_{3313}=-\frac{4\bar h_{112}\bar h_{123}}{\bar \lambda_{1}}+2\bar h_{111}\lim_{t\rightarrow\infty}\langle X, e_{3}\rangle(p_{t}),\\
&\bar H_{,23}=\bar h_{1123}+\bar h_{2223}+\bar h_{3323}= \frac{4\bar h_{111}\bar h_{123}}{\bar \lambda_{1}}+2\bar h_{112}\lim_{t\rightarrow\infty}\langle X, e_{3}\rangle(p_{t}).
\end{aligned}
\end{cases}
\end{equation}

\noindent From \eqref{3.1-8}, \eqref{3.1-13}, \eqref{3.1-43} and \eqref{3.1-44} we have
\begin{equation}\label{3.1-46}
\begin{cases}
\begin{aligned}
&\frac{1}{3}\lim_{t\rightarrow\infty}\nabla_{1}\nabla_{1}f_{3}(p_{t})=\bar \lambda^{2}_{1}(\bar h_{1111}+\bar h_{2211})-2\bar \lambda_{1}\bar h^{2}_{123},\\
&\frac{1}{3}\lim_{t\rightarrow\infty}\nabla_{2}\nabla_{2}f_{3}(p_{t})=\bar \lambda^{2}_{1}(\bar h_{1122}+\bar h_{2222})+2\bar \lambda_{1}\bar h^{2}_{123},\\
&\frac{1}{3}\lim_{t\rightarrow\infty}\nabla_{3}\nabla_{3}f_{3}(p_{t})=\bar \lambda^{2}_{1}(\bar h_{1133}+\bar h_{2233})+8\bar \lambda_{1}(\bar h^{2}_{111}-\bar h^{2}_{112}), \\
&\frac{1}{3}\lim_{t\rightarrow\infty}\nabla_{2}\nabla_{1}f_{3}(p_{t})=\bar \lambda^{2}_{1}(\bar h_{1112}+\bar h_{2212}),\\
&\frac{1}{3}\lim_{t\rightarrow\infty}\nabla_{3}\nabla_{1}f_{3}(p_{t})=\bar \lambda^{2}_{1}(\bar h_{1113}+\bar h_{2213})-4\bar \lambda_{1}\bar h_{112}\bar h_{123},\\
&\frac{1}{3}\lim_{t\rightarrow\infty}\nabla_{3}\nabla_{2}f_{3}(p_{t})=\bar \lambda^{2}_{1}(\bar h_{1123}+\bar h_{2223})+4\bar \lambda_{1}\bar h_{111}\bar h_{123},
\end{aligned}
\end{cases}
\end{equation}
and
\begin{equation}\label{3.1-47}
\begin{cases}
\begin{aligned}
&\bar \lambda_{1}(\bar h_{1111}-\bar h_{2211})=-6\bar h^{2}_{111}-2\bar h^{2}_{112}-2\bar h^{2}_{123},\\
&\bar \lambda_{1}(\bar h_{1122}-\bar h_{2222})=-2\bar h^{2}_{111}-6\bar h^{2}_{112}-2\bar h^{2}_{123},\\
&\bar \lambda_{1}(\bar h_{1133}-\bar h_{2233})=-8\bar h^{2}_{111}-8\bar h^{2}_{112}-2\bar h^{2}_{123},\\
&\bar \lambda_{1}(\bar h_{1112}-\bar h_{2212})=-8\bar h_{111}\bar h_{112},\\
&\bar \lambda_{1}(\bar h_{1113}-\bar h_{2213})=-6\bar h_{112}\bar h_{123},\\
&\bar \lambda_{1}(\bar h_{1123}-\bar h_{2223})=-6\bar h_{111}\bar h_{123}.
\end{aligned}
\end{cases}
\end{equation}

\noindent Choosing $k=l=m=1$; $k=l=2,\ m=1$; $k=l=3,\ m=1$; $k=l=1,\ m=2$; $k=l=1,\ m=3$ and $k=2,\ l=3,\ m=1$ in \eqref{3.1-38}, respectively,
we obtain
\begin{equation*}
\begin{cases}
\begin{aligned}
&3\bigg(\frac{4}{3}\lim_{t\rightarrow\infty}\nabla_{1}f_{3}(p_{t})\bar H_{,11}+\frac{4}{3}\bar H_{,1}\lim_{t\rightarrow\infty}\nabla_{1}\nabla_{1}f_{3}(p_{t})-2S\bar H_{,1}\bar H_{,11}\bigg)=0,\\

&\frac{4}{3}\lim_{t\rightarrow\infty}\nabla_{1}f_{3}(p_{t})\bar H_{,22}+\frac{4}{3}\bar H_{,1}\lim_{t\rightarrow\infty}\nabla_{2}\nabla_{2}f_{3}(p_{t})-2S\bar H_{,1}\bar H_{,22} \\
&+2\bigg(\frac{4}{3}\lim_{t\rightarrow\infty}\nabla_{2}f_{3}(p_{t})\bar H_{,12}+\frac{4}{3}\bar H_{,2}\lim_{t\rightarrow\infty}\nabla_{2}\nabla_{1}f_{3}(p_{t})-2S\bar H_{,2}\bar H_{,12}\bigg)=0,\\

&\frac{4}{3}\lim_{t\rightarrow\infty}\nabla_{1}f_{3}(p_{t})\bar H_{,33}+\frac{4}{3}\bar H_{,1}\lim_{t\rightarrow\infty}\nabla_{3}\nabla_{3}f_{3}(p_{t})-2S\bar H_{,1}\bar H_{,33}=0,\\

&\frac{4}{3}\lim_{t\rightarrow\infty}\nabla_{2}f_{3}(p_{t})\bar H_{,11}+\frac{4}{3}\bar H_{,2}\lim_{t\rightarrow\infty}\nabla_{1}\nabla_{1}f_{3}(p_{t})-2S\bar H_{,2}\bar H_{,11} \\
&+2\bigg(\frac{4}{3}\lim_{t\rightarrow\infty}\nabla_{1}f_{3}(p_{t})\bar H_{,12}+\frac{4}{3}\bar H_{,1}\lim_{t\rightarrow\infty}\nabla_{2}\nabla_{1}f_{3}(p_{t})-2S\bar H_{,1}\bar H_{,12}\bigg)=0,\\

&2\bigg(\frac{4}{3}\lim_{t\rightarrow\infty}\nabla_{1}f_{3}(p_{t})\bar H_{,13}+\frac{4}{3}\bar H_{,1}\lim_{t\rightarrow\infty}\nabla_{3}\nabla_{1}f_{3}(p_{t})-2S\bar H_{,1}\bar H_{,13}\bigg)=0,\\

&\bigg(\frac{4}{3}\lim_{t\rightarrow\infty}\nabla_{1}f_{3}(p_{t})\bar H_{,23}+\frac{4}{3}\bar H_{,1}\lim_{t\rightarrow\infty}\nabla_{3}\nabla_{2}f_{3}(p_{t})-2S\bar H_{,1}\bar H_{,23}\bigg)\\
&+\bigg(\frac{4}{3}\lim_{t\rightarrow\infty}\nabla_{2}f_{3}(p_{t})\bar H_{,13}+\frac{4}{3}\bar H_{,2}\lim_{t\rightarrow\infty}\nabla_{3}\nabla_{1}f_{3}(p_{t})-2S\bar H_{,2}\bar H_{,13}\bigg)=0.
\end{aligned}
\end{cases}
\end{equation*}
where $\bar H_{,3}=0$ and $\lim_{t\rightarrow\infty}\nabla_{3}f_{3}(p_{t})=0$.

\noindent And then, from \eqref{3.1-43}, we have
\begin{equation*}
\begin{cases}
\begin{aligned}
&\bar h_{111}(4\lim_{t\rightarrow\infty}\nabla_{1}\nabla_{1}f_{3}(p_{t})-3S\bar H_{,11})=0,\\
&\bar h_{111}(4\lim_{t\rightarrow\infty}\nabla_{2}\nabla_{2}f_{3}(p_{t})-3S\bar H_{,22})+2\bar h_{112}(4\lim_{t\rightarrow\infty}\nabla_{2}\nabla_{1}f_{3}(p_{t})-3S\bar H_{,12})=0,\\
&\bar h_{111}(4\lim_{t\rightarrow\infty}\nabla_{3}\nabla_{3}f_{3}(p_{t})-3S\bar H_{,33})=0,\\
&\bar h_{112}(4\lim_{t\rightarrow\infty}\nabla_{1}\nabla_{1}f_{3}(p_{t})-3S\bar H_{,11})+2\bar h_{111}(4\lim_{t\rightarrow\infty}\nabla_{2}\nabla_{1}f_{3}(p_{t})-3S\bar H_{,12})=0,\\
&\bar h_{111}(4\lim_{t\rightarrow\infty}\nabla_{3}\nabla_{1}f_{3}(p_{t})-3S\bar H_{,13})=0,\\
&\bar h_{111}(4\lim_{t\rightarrow\infty}\nabla_{3}\nabla_{2}f_{3}(p_{t})-3S\bar H_{,23})+\bar h_{112}(4\lim_{t\rightarrow\infty}\nabla_{3}\nabla_{1}f_{3}(p_{t})-3S\bar H_{,13})=0.
\end{aligned}
\end{cases}
\end{equation*}
Therefore,
\begin{equation}\label{3.1-48}
\begin{aligned}
&\frac{1}{3}\lim_{t\rightarrow\infty}\nabla_{1}\nabla_{1}f_{3}(p_{t})=\frac{1}{4}S\bar H_{,11},\ \ \frac{1}{3}\lim_{t\rightarrow\infty}\nabla_{2}\nabla_{2}f_{3}(p_{t})=\frac{1}{4}S\bar H_{,22},\\
&\frac{1}{3}\lim_{t\rightarrow\infty}\nabla_{3}\nabla_{3}f_{3}(p_{t})=\frac{1}{4}S\bar H_{,33},\ \ \frac{1}{3}\lim_{t\rightarrow\infty}\nabla_{2}\nabla_{1}f_{3}(p_{t})=\frac{1}{4}S\bar H_{,12},\\
 &\frac{1}{3}\lim_{t\rightarrow\infty}\nabla_{3}\nabla_{1}f_{3}(p_{t})=\frac{1}{4}S\bar H_{,13},\ \ \frac{1}{3}\lim_{t\rightarrow\infty}\nabla_{3}\nabla_{2}f_{3}(p_{t})=\frac{1}{4}S\bar H_{,23}.
\end{aligned}
\end{equation}

\noindent From \eqref{3.1-45}, \eqref{3.1-46} and \eqref{3.1-48}, we know

\begin{equation}\label{3.1-49}
\begin{cases}
\begin{aligned}
&\bar h_{1111}+\bar h_{2211}=\frac{2(\bar h^{2}_{111}-\bar h^{2}_{112}+\bar h^{2}_{123})}{\bar \lambda_{1}}+\frac{1}{2}\bar \lambda_{1},\\
 &\bar h_{1122}+\bar h_{2222}=\frac{2(\bar h^{2}_{111}-\bar h^{2}_{112}-\bar h^{2}_{123})}{\bar \lambda_{1}}-\frac{1}{2}\bar \lambda_{1},\\
&\bar h_{1133}+\bar h_{2233}=-\frac{4(\bar h^{2}_{111}-\bar h^{2}_{112})}{\bar \lambda_{1}}, \
 \bar h_{1112}+\bar h_{2212}=\frac{1}{2}\bar h_{123}\lim_{t\rightarrow\infty}\langle X, e_{3}\rangle(p_{t}),\\
&\bar h_{1113}+\bar h_{2213}=\frac{2\bar h_{112}\bar h_{123}}{\bar \lambda_{1}}+\bar h_{111}\lim_{t\rightarrow\infty}\langle X, e_{3}\rangle(p_{t}),\\
&\bar h_{1123}+\bar h_{2223}=-\frac{2\bar h_{111}\bar h_{123}}{\bar \lambda_{1}}+\bar h_{112}\lim_{t\rightarrow\infty}\langle X, e_{3}\rangle(p_{t}).
\end{aligned}
\end{cases}
\end{equation}

\noindent From \eqref{3.1-45}, \eqref{3.1-47} and \eqref{3.1-49}, we have
\begin{equation}\label{3.1-50}
\begin{cases}
\begin{aligned}
&\bar h_{1111}=-\frac{2(\bar h^{2}_{111}+\bar h^{2}_{112})}{\bar \lambda_{1}}+\frac{1}{4}\bar \lambda_{1}, \ \
 \bar h_{2211}=\frac{2(2\bar h^{2}_{111}+\bar h^{2}_{123})}{\bar \lambda_{1}}+\frac{1}{4}\bar \lambda_{1}, \\

&\bar h_{3311}=\frac{2(\bar h^{2}_{111}-\bar h^{2}_{112}-\bar h^{2}_{123})}{\bar \lambda_{1}}+\frac{1}{2}\bar \lambda_{1}, \ \
\bar h_{1122}=-\frac{2(2\bar h^{2}_{112}+\bar h^{2}_{123})}{\bar \lambda_{1}}-\frac{1}{4}\bar \lambda_{1}, \\

&\bar h_{2222}=\frac{2(\bar h^{2}_{111}+\bar h^{2}_{112})}{\bar \lambda_{1}}-\frac{1}{4}\bar \lambda_{1}, \ \
\bar h_ {3322}=\frac{2(\bar h^{2}_{111}-\bar h^{2}_{112}+\bar h^{2}_{123})}{\bar \lambda_{1}}-\frac{1}{2}\bar \lambda_{1}, \\

&\bar h_{1133}=-\frac{6\bar h^{2}_{111}+2\bar h^{2}_{112}+\bar h^{2}_{123}}{\bar \lambda_{1}}, \ \
\bar h_{2233}=\frac{2\bar h^{2}_{111}+6\bar h^{2}_{112}+\bar h^{2}_{123}}{\bar \lambda_{1}}, \\

&\bar h_{1112}=-\frac{4\bar h_{111}\bar h_{112}}{\bar \lambda_{1}}+\frac{1}{4}\bar h_{123}\lim_{t\rightarrow\infty}\langle X, e_{3}\rangle(p_{t}), \ \
\bar h_{2212}=\frac{4\bar h_{111}\bar h_{112}}{\bar \lambda_{1}}+\frac{1}{4}\bar h_{123}\lim_{t\rightarrow\infty}\langle X, e_{3}\rangle(p_{t}), \\

&\bar h_{3312}=\frac{1}{2}\bar h_{123}\lim_{t\rightarrow\infty}\langle X, e_{3}\rangle(p_{t}), \ \
\bar h_{1113}=-\frac{2\bar h_{112}\bar h_{123}}{\bar \lambda_{1}}+\frac{1}{2}\bar h_{111}\lim_{t\rightarrow\infty}\langle X, e_{3}\rangle(p_{t}), \\

&\bar h_{2213}=\frac{4\bar h_{112}\bar h_{123}}{\bar \lambda_{1}}+\frac{1}{2}\bar h_{111}\lim_{t\rightarrow\infty}\langle X, e_{3}\rangle(p_{t}), \ \
\bar h_{3313}=-\frac{6\bar h_{112}\bar h_{123}}{\bar \lambda_{1}}+\bar h_{111}\lim_{t\rightarrow\infty}\langle X, e_{3}\rangle(p_{t}),\\

&\bar h_{1123}=-\frac{4\bar h_{111}\bar h_{123}}{\bar \lambda_{1}}+\frac{1}{2}\bar h_{112}\lim_{t\rightarrow\infty}\langle X, e_{3}\rangle(p_{t}), \ \
\bar h_{2223}=\frac{2\bar h_{111}\bar h_{123}}{\bar \lambda_{1}}+\frac{1}{2}\bar h_{112}\lim_{t\rightarrow\infty}\langle X, e_{3}\rangle(p_{t}),\\

&\bar h_{3323}=\frac{6\bar h_{111}\bar h_{123}}{\bar \lambda_{1}}+\bar h_{112}\lim_{t\rightarrow\infty}\langle X, e_{3}\rangle(p_{t}).
\end{aligned}
\end{cases}
\end{equation}

\noindent By \eqref{3.1-6} and \eqref{3.1-50}, we have
\begin{equation*}
\begin{aligned}
&\bar h_{1122}-\bar h_{2211}=-\frac{4(\bar h^{2}_{111}+\bar h^{2}_{112}+\bar h^{2}_{123})}{\bar \lambda_{1}}-\frac{1}{2}\bar \lambda_{1}=-2\bar \lambda^{3}_{1}, \\
&\bar h_{1133}=\bar h_{3311}=\frac{2(\bar h^{2}_{111}-\bar h^{2}_{112}-\bar h^{2}_{123})}{\bar \lambda_{1}}+\frac{1}{2}\bar \lambda_{1}=-\frac{6\bar h^{2}_{111}+2\bar h^{2}_{112}+\bar h^{2}_{123}}{\bar \lambda_{1}}, \\
&\bar h_{2233}=\bar h_{3322}=\frac{2(\bar h^{2}_{111}-\bar h^{2}_{112}+\bar h^{2}_{123})}{\bar \lambda_{1}}-\frac{1}{2}\bar \lambda_{1}=\frac{2\bar h^{2}_{111}+6\bar h^{2}_{112}+\bar h^{2}_{123}}{\bar \lambda_{1}},
\end{aligned}
\end{equation*}
and then,
\begin{equation}\label{3.1-51}
\bar h^{2}_{111}=\bar h^{2}_{112}=\frac{1}{80}S^{2}-\frac{1}{32}S,\ \ \bar h^{2}_{123}=\frac{1}{10}S^{2}.
\end{equation}

\noindent Since
\begin{equation*}
S \geq 1, \ \ \sum_{i,j,k}h_{ijk}h_{ijkl}=0 \ \ \text{for } \ l=1, 2, 3.
\end{equation*}

\noindent By $\bar h_{111}=\bar h_{221}=\frac{1}{2}\bar h_{331}$, $\bar h_{112}=\bar h_{222}=\frac{1}{2}\bar h_{332}$ and \eqref{3.1-6},
we have,
\begin{equation}\label{3.1-52}
\begin{aligned}
&\bar h_{111}\bar h_{1111}+3\bar h_{111}\bar h_{2211}+6\bar h_{111}\bar h_{3311}+3\bar h_{112}\bar h_{1112}+\bar h_{112}\bar h_{2212}+6\bar h_{112}\bar h_{3312}\\
&+6\bar h_{123}\bar h_{1123}=0, \\
&\bar h_{111}\bar h_{1112}+3\bar h_{111}\bar h_{2212}+6\bar h_{111}\bar h_{3312}+3\bar h_{112}\bar h_{1122}+\bar h_{112}\bar h_{2222}+6\bar h_{112}\bar h_{3322}\\
&+6\bar h_{123}\bar h_{2213}=0, \\
&\bar h_{111}\bar h_{1113}+3\bar h_{111}\bar h_{2213}+6\bar h_{111}\bar h_{3313}+3\bar h_{112}\bar h_{1123}+\bar h_{112}\bar h_{2223}+6\bar h_{112}\bar h_{3323}\\
&+6\bar h_{123}\bar h_{3312}=0.
\end{aligned}
\end{equation}

\noindent From \eqref{3.1-51}, we get that either $\bar h_{111}=\bar h_{112}\neq 0$ or $\bar h_{111}=-\bar h_{112}\neq 0$ and $\bar h_{123}\neq 0$.

\noindent If $\bar h_{111}=\bar h_{112}$, from \eqref{3.1-50} and \eqref{3.1-52}, we have
\begin{equation*}
\begin{aligned}
&-\frac{30\bar h^{2}_{123}}{\bar \lambda_{1}}+4\bar \lambda_{1}+7\bar h_{123}\lim_{t\rightarrow\infty}\langle X, e_{3}\rangle(p_{t})=0, \\
& \frac{30\bar h^{2}_{123}}{\bar \lambda_{1}}-4\bar \lambda_{1}+7\bar h_{123}\lim_{t\rightarrow\infty}\langle X, e_{3}\rangle(p_{t})=0.
\end{aligned}
\end{equation*}
Then, $\lim_{t\rightarrow\infty}\langle X, e_{3}\rangle(p_{t})=0$ and $\bar h^{2}_{123}=\frac{1}{15}S$.

\noindent Since $\bar h^{2}_{123}=\frac{1}{10}S^{2}$ and $S\geq 1$, we obtain
$$S = \frac{2}{3}.$$
It is impossible.

\noindent If $\bar h_{111}=-\bar h_{112}$, from \eqref{3.1-50} and \eqref{3.1-52}, we know that
\begin{equation*}
\begin{aligned}
&-\frac{30\bar h^{2}_{123}}{\bar \lambda_{1}}+4\bar \lambda_{1}-7\bar h_{123}\lim_{t\rightarrow\infty}\langle X, e_{3}\rangle(p_{t})=0, \\
&-\frac{30\bar h^{2}_{123}}{\bar \lambda_{1}}+4\bar \lambda_{1}+7\bar h_{123}\lim_{t\rightarrow\infty}\langle X, e_{3}\rangle(p_{t})=0.
\end{aligned}
\end{equation*}
Then, $\lim_{t\rightarrow\infty}\langle X, e_{3}\rangle(p_{t})=0$ and $\bar h^{2}_{123}=\frac{1}{15}S$.
As in the previous case, it is a contradiction.

\noindent {\bf Case 2: $\bar \lambda_{1}\bar \lambda_{2} \bar \lambda_{3}\neq0$}.

\noindent From $S(p)\neq 0$, $\bar H=0$ and $f_{3}=\frac{H}{2}(3S-H^{2})+3\lambda_{1}\lambda_{2}\lambda_{3}$, we have $\lim_{t\rightarrow\infty}f_{3}(p_{t})\neq 0$.

\noindent Since $f_{4}=\frac{4}{3}Hf_{3}-H^{2}S+\frac{1}{6}H^{4}+\frac{1}{2}S^{2}$ and $\lim_{t\rightarrow\infty}f_{3}(p_{t})\neq 0$, we get
\begin{equation*}
\begin{aligned}
0=\nabla_{k}f_{4}
 =&\frac{4}{3}f_{3} H_{,k}+\frac{4}{3}H\nabla_{k}f_{3}-2SHH_{,k}+\frac{2}{3}H^{3}H_{,k}, \\
0=\nabla_{l}\nabla_{k}f_{4}
 =&\frac{4}{3}f_{3}H_{,kl}-2SHH_{,kl}
   +\frac{2}{3}H^{3}H_{,kl}+\frac{4}{3}H\nabla_{l}\nabla_{k}f_{3}+\frac{4}{3}\nabla_{l}f_{3} H_{,k} \\
  &+\frac{4}{3}H_{,l}\nabla_{k}f_{3}-2SH_{,k}H_{,l}+2H^{2}H_{,k}H_{,l},
\ \ \text{for } \ k, l=1, 2, 3.
\end{aligned}
\end{equation*}
Then, $\bar H_{,k}=0$ and $\bar H_{,kl}=0$ for $k, l=1, 2, 3$.
\noindent Especially,
\begin{equation}\label{3.1-53}
\begin{cases}
\begin{aligned}
&\bar H_{,1}=\bar \lambda_{1}\lim_{t\rightarrow\infty}\langle X,e_{1} \rangle(p_{t})=0,\\
&\bar H_{,2}=\bar \lambda_{2}\lim_{t\rightarrow\infty}\langle X,e_{2} \rangle(p_{t})=0,\\
&\bar H_{,3}=\bar \lambda_{3}\lim_{t\rightarrow\infty}\langle X,e_{3} \rangle(p_{t})=0,
\end{aligned}
\end{cases}
\end{equation}
and
\begin{equation}\label{3.1-54}
\begin{cases}
\begin{aligned}
&\bar H_{,11}=\bar h_{1111}+\bar h_{2211}+\bar h_{3311}=\sum_{k}h_{11k}\lim_{t\rightarrow\infty}\langle X,e_{k} \rangle(p_{t})+\bar \lambda_{1}=0,\\
&\bar H_{,22}=\bar h_{1122}+\bar h_{2222}+\bar h_{3322}=\sum_{k}h_{22k}\lim_{t\rightarrow\infty}\langle X,e_{k} \rangle(p_{t})+\bar \lambda_{2}=0,\\
&\bar H_{,33}=\bar h_{1133}+\bar h_{2233}+\bar h_{3333}=\sum_{k}h_{33k}\lim_{t\rightarrow\infty}\langle X,e_{k} \rangle(p_{t})+\bar \lambda_{3}=0.
\end{aligned}
\end{cases}
\end{equation}
From \eqref{3.1-53} and $\bar \lambda_{k} \neq 0$ for $k=1,2,3$, one has
\begin{equation*}
\lim_{t\rightarrow\infty}\langle X,e_{k} \rangle(p_{t})=0, \ \  \text{for} \ k=1, 2, 3.
\end{equation*}
And then, by \eqref{3.1-54}, we know that $\bar \lambda_{k}=0$ for $k=1, 2, 3$.
It is a contradiction.

\end{proof}

\vskip3mm
\noindent
{\it Proof of Theorem \ref{theorem 1}}. If $S=0$, then $X: M^{3}\to \mathbb{R}^{4}$ is $\mathbb{R}^3$. If $S \neq 0$, from Theorem \ref{theorem 2}, we know $\inf H^{2}>0$. By using of the Lemma 2.4, we conclude
$X: M^{3}\to \mathbb{R}^{4}$ is either $S^{1}(1)\times \mathbb{R}^{2}$, $S^{2}(\sqrt{2})\times \mathbb{R}^{1}$ or  $S^{3}(\sqrt{3})$.
\begin{flushright}
$\square$
\end{flushright}


\begin{thebibliography}{99}
\bibitem{AL}
U. Abresch and J. Langer, The normalized curve shortening flow and
homothetic  solutions, J. Differential Geom., {\bf 23}(1986), 175-196.

\bibitem{A}
H. Anciaux, Construction of Lagrangian self-similar solutions to the mean curvature flow in  $\mathbb C^n$,  Geom. Dedic., {\bf 120}(2006), 37-48.

\bibitem{B}
S. Brendle,  Embedded self-similar shrinkers of genus 0,   Ann. of Math.,  {\bf 183}(2016), 715-728.

\bibitem{CL}
H.-D. Cao and H. Li,  A gap theorem for self-shrinkers of the mean
curvature flow in arbitrary codimension, Calc. Var. Partial Differential Equations, {\bf 46} (2013), 879-889.

\bibitem{CHW}
Q. -M. Cheng,  H. Hori and G. Wei,  Complete Lagrangian self-shrinkers in  $\mathbb R^{4}$ , arXiv:1802.02396.

\bibitem{CO}
Q. -M. Cheng and S. Ogata, $2$-dimensional complete  self-shrinkers in $\mathbb R^{3}$,
  Math. Z., {\bf 284}(2016), 537-542.

\bibitem{CP}
Q. -M. Cheng and Y. Peng,  Complete  self-shrinkers of the mean curvature flow,
  Calc. Var. Partial Differential Equations, {\bf 52} (2015), 497-506.

\bibitem{CW}
Q. -M. Cheng and G. Wei,  A gap theorem for self-shrinkers,
Trans. Amer. Math. Soc., {\bf 367} (2015), 4895-4915.

\bibitem{CZ}
X. Cheng and D. Zhou,  Volume estimate about shrinkers, Proc. Amer. Math. Soc., {\bf 141} (2013), 687-696.

\bibitem{CM}
T. H. Colding and W. P.  Minicozzi II,  Generic mean curvature flow I;  Generic singularities,
Ann. of Math.,  {\bf 175} (2012), 755-833.

\bibitem{DX1}
Q. Ding and Y. L. Xin, Volume growth, eigenvalue and compactness for self-shrinkers,  Asian J. Math., {\bf 17} (2013), 443-456.

\bibitem{DX2}
Q. Ding and Y. L. Xin, The rigidity theorems of self shrinkers, Trans. Amer. Math. Soc., {\bf 366} (2014), 5067-5085.

\bibitem{H}
H. Halldorsson, Self-similar solutions to the curve shortening flow, Trans. Amer. Math. Soc., {\bf 364} (2012), 5285-5309.

\bibitem{H1}
G. Huisken,
Flow by mean curvature convex surfaces into spheres, J. Differential
Geom.,  {\bf 20} (1984), 237-266.

\bibitem{H2} G. Huisken,
Asymptotic behavior for singularities of the mean curvature flow, J. Differential
Geom., {\bf 31} (1990), 285-299.

\bibitem{H3}
G. Huisken, Local and global behaviour of hypersurfaces moving by mean curvature,  Differential
geometry: partial differential equations on manifolds (Los Angeles, CA, 1990), Proc. Sympos. Pure
Math., {\bf 54}, Part 1, Amer. Math. Soc., Providence, RI, (1993), 175-191.

\bibitem{L}
H. B. Lawson,  Local rigidity theorems for minimal hypersurfaces,
Ann. of Math.,  {\bf 89} (1969), 187-197.


\bibitem{[LS]}
Nam Q. Le and N. Sesum, Blow-up rate of the mean curvature during the mean curvature
flow and a gap theorem for self-shrinkers, Comm. Anal. Geom., {\bf 19} (2011), 1-27.

\bibitem{[LXX]}
L. Lei, H. W. Xu and Z. Y. Xu,  A new pinching theorem for complete self-shrinkers and its generalization, arXiv:1712.01899v1.

\bibitem{LW}
H. Li and X. F.  Wang, New characterizations of the Clifford torus as a Lagrangian  self-shrinkers,
J. Geom. Anal., {\bf 27} (2017), 1393-1412.

\bibitem{LW1}
H. Li and Y.  Wei, {\it Lower volume growth estimates for self-shrinkers of mean curvature flow}, Proc. Amer. Math. Soc.,  {\bf 142} (2014), 3237-3248.

\bibitem{LW2}
H. Li and Y. Wei, {\it Classification and rigidity of self-shrinkers in the mean curvature flow}, J. Math. Soc. Japan, {\bf 66} (2014), 709-734.

\bibitem{N1}
A. Neves, Singularities of Lagrangian mean curvature flow: monotone case,  Math. Res. Lett., {\bf 17}(2010), 109-126.

\bibitem{Y}
S. T. Yau, Submanifolds with Constant Mean Curvature I ,  Amer.  J.   Math., {\bf 96} (1974), 346-366.

\end{thebibliography}
\end{document}